\documentclass[reqno]{amsart}
  \usepackage{mathabx}
\usepackage{cite}
\usepackage{color}
\allowdisplaybreaks

\theoremstyle{plain}
\newtheorem{thm}{Theorem}[section]

\newtheorem{lem}[thm]{Lemma}
\newtheorem{prop}[thm]{Proposition}

\theoremstyle{definition}

\theoremstyle{remark}
\newtheorem{rem}{Remark}[section]
\numberwithin{equation}{section}

\newcommand{\dv}{{\rm div\,}}
\newcommand{\dif}{{\rm d}}

\newcommand{\p}{{\partial}}
\newcommand{\pa}{{\partial^\alpha_x}}

\renewcommand{\u}{{\bf u}}
\renewcommand{\j}{{\bf{I}}_1}
\newcommand{\G}{{{\bf{J}}_1^\epsilon}}
\newcommand{\F}{{J^\epsilon_0}}
\newcommand{\I}{{\bf I}}
\renewcommand{\v}{{\vvvert}}

\newcommand{\U}{{\mathbf U}}

\newcommand{\W}{{\mathbf W}}
\newcommand{\D}{{\mathbf D}}
\newcommand{\A}{{\mathbf A}}
\newcommand{\s}{{\mathbf S}}
\newcommand{\R}{{\mathbf R}}

\begin{document}

\title[ Navier-Stokes-Fourier-P1 approximation model]
 {Non-relativistic limit of the compressible Navier-Stokes-Fourier-P1 approximation model arising in radiation hydrodynamics}
\author{Song Jiang}
\address{Institute of Applied Physics and Computational Mathematics, P.O. Box 8009, Beijing 100088, P.R. China}
\email{jiang@iapcm.ac.cn}
%
\author [Fucai Li] {Fucai Li}%
\address{Department  of Mathematics,  Nanjing University, Nanjing 210093, P.R. China}
 \email{fli@nju.edu.cn}

\author[Feng Xie]{Feng Xie}
\address{Department of Mathematics, and LSC-MOE,
 Shanghai Jiao Tong University,
Shanghai 200240, P.R.China}
\email{tzxief@sjtu.edu.cn}

\begin{abstract}
As is well-known that the general radiation hydrodynamics models include two mainly coupled parts: one is macroscopic
 fluid part, which is governed by the compressible Navier-Stokes-Fourier equations; another is radiation field part, which is described by the transport equation of photons. Under the two physical approximations: ``gray" approximation and P1 approximation, one can derive the so-called Navier-Stokes-Fourier-P1 approximation radiation hydrodynamics model from the general one. In this paper we study the non-relativistic limit problem for the Navier-Stokes-Fourier-P1 approximation model due to the fact that the speed of light is much larger than the speed of the  macroscopic fluid.
Our results give a rigorous derivation of the widely used macroscopic model in radiation hydrodynamics.
\end{abstract}

\keywords{Radiation Hydrodynamics, Compressible Navier-Stokes-Fourier system, P1 approximation, ``Gray" approximation, Non-relativistic limit}

\subjclass[2000]{35Q30, 35Q70, 35B25}

\maketitle

\section{Introduction and Main Results} \label{S1}
The key aim of the radiation hydrodynamics is to include the
radiation effects into hydrodynamics. And hydrodynamics with explicit account of the radiation energy and momentum contributions
constitutes the main charter of ``radiation hydrodynamics". If the viscosity and heat-conductivity of macroscopic fluids are also
considered, the general radiation hydrodynamics equations can be written as the following compressible Navier-Stokes-Fourier
system with additional radiation terms (\!\cite{P}):
\begin{align}
&\p_t\rho  +\dv(\rho\u)=0,  \label{0.1}\\
&\p_t\Big(\rho\u+\frac{1}{c^2}F_r\Big) +  \dv(\rho\u\otimes \u + P \mathbb{I}_{n}+P_r) = \dv\Psi(\u), \label{0.2} \\
&\p_t(\rho E+E_r) +\dv (\rho\u(E+P)+F_r)=\dv (\Psi(\u)\cdot \u)+\dv(\kappa \nabla \theta).\label{0.3}
   \end{align}
 Here  the unknowns $\rho,\u=(u_1,\dots,u_n)\in \mathbb{R}^n (n=2,3)$
and $\theta$  denote the density, the velocity, and the temperature of the fluid, respectively;
$\Psi(\u)$ is the viscous stress tensor given by
\begin{equation}\label{psi}
\Psi(\u)=2\mu \mathbb{D}(\u)+\lambda\dv\u \;\mathbb{I}_n, \quad \mathbb{D}(\u)=(\nabla\u+\nabla\u^\top)/2,
\end{equation}
where $\mathbb{I}_n$  denotes the $n\times n$ identity matrix,
and $\nabla \u^\top$   the transpose of the matrix
$\nabla \u$. The pressure $P=P(\rho,\theta)$ and
the internal energy $e=e(\rho,\theta)$ are smooth functions of $\rho$ and $\theta$ and satisfy the Gibbs relation
\begin{equation}\label{gibbs}
\theta \mathrm{d}S=\mathrm{d}e + P\,\mathrm{d}\left(\frac{1}{\rho}\right)
\end{equation}
for some smooth function (entropy) $S=S(\rho,\theta)$, which expresses the first law of the
thermodynamics.
$E=e+\frac{|\u|^2}{2}$ denotes the total energy.
The viscosity coefficients $\mu$ and $\lambda $ of the fluid satisfy
 $\mu>0$ and $2\mu+n\lambda>0$. The parameter  $\kappa>0$ is the heat
conductivity.
 For simplicity we assume that $\mu,\lambda$, and $\kappa$ are constants.

Now we consider
 the radiation energy $E_r$, the radiation flux $F_r$ and the radiation pressure $P_r$ appearing in \eqref{0.1}--\eqref{0.3} which can be defined by
\begin{align*}
& E_r=\frac{1}{c}\int_0^{\infty}\dif\nu\int_{S^{n-1}}I(\nu,\omega)\dif \omega,\\
&F_r=\int_0^{\infty}\dif\nu\int_{S^{n-1}}\omega I(\nu,\omega)\dif\omega,\\
& P_r=\frac{1}{c}\int_0^{\infty}\dif\nu\int_{S^{n-1}}\omega\otimes\omega I(\nu,\omega)\dif\omega.
\end{align*}
Here the radiation intensity $I=I(t,x,\nu,\omega)$, depending on the direction vector $\omega\in S^{n-1}$ and the frequency $\nu\geq 0$,
 is determined by solving the linear Boltzmann type equation:
\begin{align}
\frac1c\p_t I +\omega\cdot\nabla I=&S(\nu)-\sigma_a(\nu)I\nonumber\\
&+\int_0^{\infty}\dif\nu'\int_{S^{n-1}}\Big[\frac{\nu}{\nu'}\sigma_s(\nu'\rightarrow\nu)I(\nu',\omega')
-\sigma_s(\nu\rightarrow\nu')I(\nu,\omega)\Big]\dif\omega'. \nonumber
\end{align}
The emission term $S(\nu)$ can be taken as the well-known Planck function, i.e.,
 $$S(\nu)=2h\nu^3c^{-2}(e^{h\nu/k\theta}-1)^{-1}.$$
 In general the absorbing coefficient $\sigma_a$ and the scattering coefficient $\sigma_s$ depend on the frequency $\nu$, the density $\rho$, and the
 temperature $\theta$ of the macroscopic fluid.

 In the present paper,
  we focus on the ``gray" approximation case such that the transport coefficients $\sigma_a$ and $\sigma_s$ are independent of the frequency $\nu$.
  Consequently, the radiation quantities both $I$ and $ S$ can be integrated on frequency. For $S$, we have
\begin{align*}
\int_0^{\infty}S(\nu)\dif\nu=\int_0^{\infty}2h\nu^3c^{-2}(e^{h\nu/k\theta}-1)^{-1}\dif\nu=\bar{C}\theta^4,
\end{align*}
for some positive constant $\bar{C}$. To deal with $I$, for simplification, we further assume that both $\sigma_a$ and $\sigma_s$
are two positive constants in the following derivation. It should be noted that the general case $\sigma_a=\sigma_a(\rho,\theta)$ and $\sigma_s=\sigma_s(\rho,\theta)$ can also be dealt with similarly. In this way, the equation for the integration of $I$
with respect to frequency $\nu$, still denoted by $I$, can be written as
\begin{align}
\label{FX}
\frac1c\p_tI+\omega\cdot\nabla I=\bar{C}\theta^4-\sigma_a I+\sigma_s|S^{n-1}|(\langle I\rangle-I)
\end{align}
with the average $\langle I\rangle:=\frac{1}{|S^{n-1}|}\int_{S^{n-1}}I(t,x, \omega)\dif\omega$.

In addition, when the distribution of photons is almost isotropic, one can take the P1 hypothesis by choosing the ansatz
\begin{align}
\label{FX1}
I=I_0+\I_1\cdot\omega,
\end{align}
where $I_0$ and $\I_1$ do not depend on $\omega$, $\I_1\cdot\omega$ is regarded as a correction term of the main term $I_0$.
Inserting the ansatz (\ref{FX1}) into (\ref{FX}) gives
\begin{align}
\label{FX2}
&\frac1c\p_t(I_0+\I_1\cdot\omega)+\omega\cdot\nabla (I_0+\I_1\cdot\omega)\nonumber\\
&\qquad =\bar{C}\theta^4-\sigma_a (I_0+\I_1\cdot\omega)+\sigma_s|S^{n-1}|(\langle I_0+\I_1\cdot\omega\rangle-(I_0+\I_1\cdot\omega)).
\end{align}
Then integrating both the equation (\ref{FX2}) and the resulting equation of (\ref{FX2})$\cdot\omega$ with respect to $\omega$ over $S^{n-1}$ lead to
\begin{align}
&\frac1c\p_t I_0+\frac{1}{n|S^{n-1}|}\dv_x \I_1=\bar{C}\theta^4-\sigma_a I_0,\label{FX3}\\
&\frac1c\p_t \I_1+\nabla_xI_0=-(\sigma_a+\sigma_s|S^{n-1}|)\I_1. \label{FX4}
\end{align}
Moreover, by the equation (\ref{FX}) and the definitions of $E_r, F_r$ and $P_r$, we obtain that
\begin{align}
&\frac{1}{c^2}\p_t F_r+\dv P_r=-\frac1c\frac{\sigma_a+\sigma_s|S^{n-1}|}{n}\I_1, \label{FX5}\\
&\p_t E_r+\dv F_r=\bar{C}|S^{n-1}|\theta^4-|S^{n-1}|\sigma_aI_0.\label{FX6}
\end{align}
One can also refer to the Chapter 3 in the book \cite{P}
by Pomraning for the above derivation in details.

In general, the speed of light $c$ can be regarded as a very large number such that the reciprocal of the light speed $\frac1c$ is very small.
If we take $\epsilon=\frac1c$ and ignore the influence of other constants, we obtain the following compressible Navier-Stokes-Fourier-P1 approximation model  via the system
\eqref{0.1}--\eqref{0.3}, \eqref{FX3}--\eqref{FX6}:
\begin{align}
& \p_t\rho  +\dv(\rho\u)=0, \label{rada} \\
&\p_t(\rho\u) +  \dv(\rho\u\otimes \u + P \mathbb{I}_{n}) = \dv\Psi(\u)+\epsilon \j, \label{radb} \\
&\p_t(\rho E) +\dv (\rho\u(E+P))=\dv (\Psi(\u)\cdot \u)+\dv(\kappa \nabla \theta)+I_0-\theta^4,\label{radc}\\
&\epsilon \p_t I_0+\dv \j =\theta^4-I_0, \label{radd}\\
&\epsilon \p_t\j+\nabla I_0=-\j. \label{rade}
\end{align}
In this paper we consider the non-relativistic limit $\epsilon\rightarrow  0$ for the system \eqref{rada}--\eqref{rade}.
Formally, letting $\epsilon= 0$ in \eqref{radd} and \eqref{rade}, we obtain that
\begin{align*}
\dv \j=\theta^4-I_0, \quad -\j=\nabla I_0.
\end{align*}
Hence we have
\begin{align}\label{radf}
-\Delta I_0=\theta^4-I_0.
\end{align}
Taking gradient to \eqref{radf}, one gets
\begin{align}\label{radg}
   -\nabla \dv(\nabla I_0)=\nabla \theta^4-\nabla I_0.
\end{align}
Setting $\mathbf{q}=-\nabla I_0$, we can rewrite \eqref{radg} as
\begin{align}
   -\nabla \dv \mathbf{q} +\mathbf{q}+\nabla \theta^4=0. \nonumber
\end{align}
Therefore, we can formally obtain the following limit system from \eqref{rada}--\eqref{rade} as $\epsilon \rightarrow 0$:
\begin{align}
& \p_t\rho  +\dv(\rho\u)=0, \label{radi} \\
&\p_t(\rho\u) +  \dv(\rho\u\otimes \u + P \mathbb{I}_{n}) = \dv\Psi(\u), \label{radj} \\
&\p_t(\rho E) +\dv (\rho\u(E+P))=\dv (\Psi(\u)\cdot \u)+\dv(\kappa \nabla \theta)-\dv \mathbf{q},\label{radk}\\
& -\nabla \dv \mathbf{q} +\mathbf{q}+\nabla \theta^4=0. \label{radl}
\end{align}

The system \eqref{radi}--\eqref{radl} with $\mu, \lambda$ and $\kappa$ being zero are widely used in \cite{KNY,KNY1,LMS,LCG,L,NPZ,RX,WJX1}
to describe the dynamics of the fluid in radiation hydrodynamics. For the case that $\mu, \lambda$ and $\kappa$ are not zero,
one can refer to \cite{WJX,WX} and references cited therein.

The purpose of this paper is to give a  rigorous derivation of the system  \eqref{radi}--\eqref{radl} from the
 Navier-Stokes-Fourier-P1 approximation model \eqref{rada}--\eqref{rade} as  $\epsilon$ tends to
  zero.
 For the sake of simplicity and clarity of presentation, we shall focus on the  fluids obeying the perfect gas relations：
\begin{align} \label{rhpg}
P=\mathfrak{R}\rho \theta,\quad e=c_V\theta,
\end{align}
where the parameters $\mathfrak{R}>0$ and $ c_V\!>\!0$ are the gas constant and
the heat capacity at constant volume.
 We consider the system \eqref{rada}--\eqref{rade} in the whole space $\mathbb{R}^n$ or  the torus
$\mathbb{T}^n=(\mathbb{R}/(2\pi \mathbb{Z}))^n$, which will be denoted by $\Omega$.


In what follows, for simplicity of presentation, we take the physical constants $\mathfrak{R} $ and $ c_V$ to be one.
To emphasize the unknowns depending  on the small parameter
$\epsilon$, we rewrite the system \eqref{rada}--\eqref{rade}, \eqref{psi}, \eqref{gibbs}, \eqref{rhpg}  as
\begin{align}
 &\partial_t \rho^\epsilon  +\dv(\rho^\epsilon\u^\epsilon)=0, \label{radya} \\
&\rho^\epsilon(\p_t\u^\epsilon +  \u^\epsilon\cdot \nabla\u^\epsilon)+\nabla (\rho^\epsilon\theta^\epsilon)
  = \dv\Psi(\u^\epsilon)+ \epsilon\mathbf{I}_1^\epsilon, \label{radyb} \\
&\rho^\epsilon  (\p_t\theta^\epsilon+\u^\epsilon\cdot \nabla \theta^\epsilon)+\rho^\epsilon\theta^\epsilon  \dv \u^\epsilon   =\kappa\Delta  \theta^\epsilon
 +\Psi(\u^\epsilon):\nabla\u^\epsilon + I^\epsilon_0-(\theta^\epsilon)^4,\label{radyc}\\
&\epsilon \p_t I_0^\epsilon+\dv \mathbf{I}_1^\epsilon =(\theta^\epsilon)^4-I_0^\epsilon, \label{radyd}\\
&\epsilon \p_t\mathbf{I}_1^\epsilon+\nabla I_0^\epsilon=-\mathbf{I}_1^\epsilon, \label{radye}
\end{align}
where  $\Psi(\u^\epsilon)$ 
is defined through \eqref{psi} with $\u$ replaced by $\u^\epsilon$.
The symbol $\Psi(\u^\epsilon):\nabla\u^\epsilon$   denotes the scalar product of two matrices:
\begin{align}\label{psiu}
\Psi(\u^\epsilon):\nabla\u^\epsilon
&=\sum^n_{i,j=1}\frac{\mu}{2}\left(\frac{\partial
u^\epsilon_i}{\partial x_j} +\frac{\partial u^\epsilon_j}{\partial
x_i}\right)^2+\lambda|\dv\u^\epsilon|^2\nonumber
\\
& =
2\mu|\mathbb{D}(\u^\epsilon)|^2+\lambda|\mbox{tr}\mathbb{D}(\u^\epsilon)|^2.
\end{align}
 The system \eqref{radya}--\eqref{psiu} are supplemented with initial data
\begin{align}\label{radyf}
 (\rho^\epsilon, \u^\epsilon, \theta^\epsilon, I^\epsilon_0,\mathbf{I}_1^\epsilon)|_{t=0}
 =( \rho_0^\epsilon(x), \u_0^\epsilon(x),\theta_0^\epsilon(x), I^\epsilon_{00}(x),\mathbf{I}_{10}^\epsilon(x)), \quad x\in \Omega.
\end{align}

We also rewrite the limit equations \eqref{radi}--\eqref{radl}, \eqref{psi}, \eqref{gibbs}, \eqref{rhpg} (recall that $\mathfrak{R}=c_V=1  $) as
\begin{align}
& \p_t\rho^0  +\dv(\rho^0\u^0)=0, \label{radza} \\
&\rho^0(\p_t\u^0  +   \u^0\cdot\nabla \u^0) + \nabla (\rho^0\theta^0) = \dv\Psi(\u^0), \label{radzb} \\
&\rho^0 (\p_t\theta^0+\u^0\cdot \nabla \theta^0)+\rho^0\theta^0  \dv \u^0=\kappa\Delta  \theta^0
 +\Psi(\u^0):\nabla\u^0-\dv \mathbf{q}^0,\label{radzc}\\
& -\nabla \dv \mathbf{q}^0 +\mathbf{q}^0+\nabla (\theta^0)^4=0. \label{radzd}
\end{align}
where  $\Psi(\u^0)$ and $\Psi(\u^0):\nabla\u^0$
are defined through \eqref{psi} and \eqref{psiu} with $\u$ and $\u^\epsilon$ replaced by $\u^0$, respectively.
The system  \eqref{radza}--\eqref{radzd} are  equipped with   initial data
\begin{align}\label{radze}
(\rho^0, \u^0, \theta^0)|_{t=0}
 =( \rho_0^0(x), \u_0^0(x),\theta_0^0(x)), \quad x\in \Omega .
\end{align}

We first state a result on the local existence  of smooth solutions to the problem \eqref{radza}--\eqref{radze}, one can refer to \cite{WX} for a similar proof in details.

 \begin{prop} \label{Pa}
  Let $s>n/2+2$ be an integer and
 assume that the initial data $(\rho^0_0, \u^0_0,\theta^0_0)$ satisfy
\begin{gather*}
 \rho^0_0, \u^0_0, \theta^0_0\in H^{s+2}(\Omega), \nonumber\\
    0<\bar \rho= \inf_{x\in \Omega }\rho^0_0(x)\leq \rho^0_0(x)\leq
    \bar {\bar \rho}= \sup_{x\in \Omega }\rho^0_0(x)<+\infty,\\
    0<\bar \theta= \inf_{x\in \Omega }\theta^0_0(x)\leq \theta^0_0(x)\leq
    \bar {\bar \theta}= \sup_{x\in \Omega }\theta^0_0(x)<+\infty
\end{gather*}
 for some positive constants $\bar \rho, \bar{\bar\rho}, \bar \theta$, and $\bar{\bar\theta}$. Then there exist positive
 constants $T_*\,($the maximal time interval, $ 0<T_*\leq +\infty )$, and $\hat \rho, \tilde{\rho}, \hat \theta, \tilde{\theta} $,
 such that the problem
\eqref{radza}--\eqref{radze} has a unique classical solution $(\rho^0,\u^0,\theta^0,\mathbf{q}^0)$ satisfying
\begin{gather*}
   \rho^0 \in C^l([0,T_*),H^{s+2-l}(\Omega )),\, \,\, \u^0, \theta^0 \in  C^l([0,T_*),H^{s+2-2l}(\Omega )), \ \   l=0,1;\ \ \\
    \mathbf{q}^0 \in C^l([0,T_*),H^{s+1-2l}(\Omega )), \ \   l=0,1;\ \  \\
    0<\hat \rho= \inf_{(x,t)\in \Omega \times [0,T_*)}\rho^0(x,t)\leq  \rho^0(x,t)\leq
    {\tilde \rho}= \sup_{(x,t)\in \Omega \times [0,T_*)}\rho^0(x,t)<+\infty,\\
       0<\hat \theta= \inf_{(x,t)\in \Omega \times [0,T_*)}\theta^0(x,t)\leq \theta^0(x,t)\leq
    {\tilde \theta}= \sup_{(x,t)\in \Omega \times [0,T_*)}\theta^0(x,t)< +\infty.
\end{gather*}

 \end{prop}
Our convergence results   can be stated as follows.
\begin{thm}\label{th}
Let $s> n/2+2$ be an integer and $(\rho^0, \u^0,\theta^0, \mathbf{q}^0)$
be the unique classical solution to the problem \eqref{radza}--\eqref{radze} given in Proposition \ref{Pa}.
 Suppose that the initial data $(\rho^\epsilon_0, \u^\epsilon_0,\theta_0^\epsilon, I^\epsilon_{00},\mathbf{I}_{10}^\epsilon)$ satisfy
$$
 \rho^\epsilon_0, \u^\epsilon_0,\theta_0^\epsilon,  I^\epsilon_{00},\mathbf{I}_{10}^\epsilon\in H^{s}(\Omega ), \
 \  \inf_{x\in \Omega }
   \rho^\epsilon_0(x)>0, \      \inf_{x\in \Omega }
  \theta^\epsilon_0(x)>0,
 $$
and
\begin{align}
&  \Vert (\rho^\epsilon_0-\rho^0_0, \u^\epsilon_0-\u^0_0, \theta^\epsilon_0-\theta^0_0) \Vert_{s}\nonumber\\
& \qquad \qquad \qquad
+  \sqrt{\epsilon}\left\Vert (I^\epsilon_{00}-(-\Delta)^{-1}\dv\mathbf{q}^0_0, \mathbf{I}_{10}^\epsilon-  \mathbf{q}^0_0 )\right\Vert_{s} \leq  L_0 {\epsilon} \label{ivda}
 \end{align}
for some constant $L_0>0$. Then, for any $T_0\in (0,T_* )$, there exist a constant $L>0$, and
a sufficient small constant $\epsilon_0>0$, such that for any $\epsilon\in
(0,\epsilon_0]$, the problem \eqref{radya}--\eqref{radyf} has a unique smooth solution $(\rho^\epsilon,
\u^\epsilon, \theta^\epsilon, I^\epsilon_{0},\mathbf{I}_{1}^\epsilon)$ on $[0,T_0]$ enjoying
\begin{align}\label{iivda}
&  \Vert (\rho^\epsilon-\rho^0, \u^\epsilon-\u^0, \theta^\epsilon-\theta^0) \Vert_{s} \nonumber\\
& \qquad \qquad \qquad +  \sqrt{\epsilon}\left\Vert (I^\epsilon_{0}-(-\Delta)^{-1}\dv\mathbf{q}^0, \mathbf{I}_{1}^\epsilon-\mathbf{q}^0 )\right\Vert_{s}
\leq L {\epsilon},  \ \ t\in [0,T_0].
 \end{align}
 Here $\mathbf{q}^0_0$ is defined via the initial datum $\theta^0_0$ in the following way:
 \begin{align*}
    \mathbf{q}^0_0=\Big(\frac{-\Delta}{I-\Delta}-I\Big)\nabla (\theta^0_0)^4
 \end{align*}
 and $\|\cdot\|_{s}$ denotes the norm of Sobolev space $H^s(\Omega )$.
\end{thm}

\begin{rem}
If the domain $\Omega$ in the Proposition \ref{Pa} is the whole space $\mathbb{R}^n$, then the conditions $\rho^0_0, \u^0_0, \theta^0_0\in H^{s+2}(\Omega )$ should be replaced by $\rho^0_0-\check{\rho}, \u^0_0, \theta^0_0-\check{\theta}\in H^{s+2}(\Omega)$ for some positive constants $\check{\rho}$ and $\check{\theta}$. At the same time, the conditions $\rho^\epsilon_0, \u^\epsilon_0,\theta_0^\epsilon,  I^\epsilon_{00},
\mathbf{I}_{10}^\epsilon\in H^{s}(\Omega )$ in Theorem \ref{th} are required to be changed into $\rho^\epsilon_0-\check{\rho}, \u^\epsilon_0,\theta_0^\epsilon-\check{\theta},  I^\epsilon_{00}-\check{\theta}^4,\mathbf{I}_{10}^\epsilon\in H^{s}(\Omega )$ accordingly.
The corresponding proof is essentially unchanged and can be modified in a direct way.
\end{rem}
\begin{rem}
As a consequence of our result, we  obtain the local existence of solutions to the primitive system \eqref{radya}--\eqref{radye},
 and the life-span of which is  independent of $\epsilon$.
 Furthermore, the inequality \eqref{iivda} implies that the sequences $(\rho^\epsilon, \u^\epsilon,\theta^\epsilon)$
  converge strongly to $(\rho^0,\u^0,\theta^0)$ in $L^\infty(0,T; H^{s}(\Omega ))$ and
  $(I^\epsilon_{0}, \mathbf{I}_{1}^\epsilon)$ converge strongly
  to \linebreak
  $((-\Delta)^{-1}\dv\mathbf{q}^0,\mathbf{q}^0)$ in $L^\infty(0,T; H^{s}(\Omega ))$  but with different convergence rates.
\end{rem}

\begin{rem}
  In the local existence for the problem \eqref{radza}--\eqref{radze}, the regularity assumption on initial data
  $(\rho^0_0,\u^0_0,\theta^0_0,\mathbf{q}^0_0)\in H^s(\Omega )$ for $s>n/2+2$ is in fact sufficient.
  Here we have added more regularity assumption in Proposition \ref{Pa} in order to
  obtain more regular solutions which are needed in the proof of Theorem \ref{th}.
\end{rem}

\begin{rem}
\label{rek1.4}
The viscosity and heat conductivity terms
in the system \eqref{radya}--\eqref{radye} play a crucial role in our uniformly bounded  estimates
(in order to control some undesirable higher-order terms).
For the case  of $\lambda=\mu=\kappa=0$, we should use the quasilinear symmetric structure of the
hyperbolic part in (\ref{raderror1})--(\ref{raderror4}). That is, the positively definite matrix $\A^\epsilon$ in (\ref{FXFXFX}) is essentially used in the energy estimation.
By using arguments similar to those in showing boundedness of high norms in Chapter 2 in \cite{Majda}, the main Theorem \ref{th} is also valid for the Euler-P1 system. Precisely speaking, the smooth solutions to
(\ref{radya})--(\ref{radye}) with $\mu=\lambda=\kappa=0$ converge to the smooth solutions to (\ref{radza})--(\ref{radzd}) with $\mu=\lambda=\kappa=0$. The limit equations (\ref{radza})--(\ref{radzd}) with $\mu=\lambda=\kappa=0$ are indeed considered in  \cite{KNY,KNY1,LMS,LCG,L,RX,WJX,WJX1}. In this case the local existence of solutions to limit system is referred to \cite{KNY,KNY1}. Here we mention the related nonrelativistic limit process for a simplified model in radiation hydrodynamics in \cite{RY}.

\end{rem}

We give some comments on the proof of Theorem \ref{th}.
 The main difficulty in dealing with our non-relativistic limit is the oscillatory behavior of $I^\epsilon_0$ and $\mathbf{I}^\epsilon_1$.
The time derivatives of $I_0^\epsilon$ and $\mathbf{I}_1^\epsilon$ in \eqref{radyd}--\eqref{radye} are multiplied by a small parameter, hence the uniform energy estimates are obtained from the relaxation terms rather than from the time-derivative terms.
Besides the singularity in \eqref{radyd}--\eqref{radye},
there exists an extra singularity caused by the coupling of
$I^\epsilon_0$ and $\mathbf{I}^\epsilon_1$  in the
momentum and temperature equations.
 In this paper, we shall overcome all these difficulties by adopting and modifying the elaborate nonlinear energy method developed in
\cite{JL,JL2}.
First, we derive the error system \eqref{raderror1}--\eqref{raderror4} by  utilizing  the original system  \eqref{radya}--\eqref{radye}
and the limit system \eqref{radza}--\eqref{radzd}. In this step, we need to find the suitable quantities from the limit system, which are related to $I^\epsilon_0$ and $\mathbf{I}^\epsilon_1$. Next, we study the estimates of $H^s$-norm to the error system.
 To do so, we shall make full use of  the special structure of the error system, the Sobolev imbedding and the Moser-type
inequalities, and the regularity of the limit equations. In particular, a very refined analysis is carried out to
deal with the higher order nonlinear terms in the system \eqref{raderror1}--\eqref{raderror4}. It is noted that the damping terms in equations \eqref{raderror3}--\eqref{raderror4} also play a crucial role in controlling the nonlinear coupled terms.
Finally, we combine these obtained estimates and apply the Gronwall inequality to get the desired results.
In addition, we should remark that for fixed $\epsilon$,
the global in time existence of solutions to the barotropic case of the equations \eqref{radya}--\eqref{radye}
is achieved in the critical Besov spaces by Danchin and Ducomet recently in \cite{RB}.
 As is pointed out in \cite{MM,P} that the energy exchange between the hydrodynamics and the radiation field sometime plays a leading role.
This is the key reason why we include the energy equation into the system \eqref{radya}--\eqref{radye}. In this way it will greatly increase difficulties of mathematical analysis.  To our best knowledge, for fixed $\epsilon$ the global existence of strong solutions to the equations \eqref{radya}--\eqref{radye} is still open, because the analysis of the spectrum for the linearized system is very complicated and is left for our future study.

\medskip
Before ending this introduction, we give some notations and recall some basic facts which
will be frequently used throughout this paper.

(1) We denote by $\langle \cdot,\cdot\rangle$ the standard inner product in $L^2(\Omega )$
with $\langle f,f\rangle=\|f\|^2$, by
$H^k$ the standard Sobolev space $W^{k,2}$ with norm $\|\cdot\|_{k}$.  The notation $\|(A_1,A_2, \dots,
 A_l)\|_k$ means the summation of $\|A_i\|_k$ from $i=1$ to $i=l$.
For a multi-index $\alpha = (\alpha_1,  \dots, \alpha_n)$,  we denote
$\partial_x^\alpha =\partial^{\alpha_1}_{x_1}\cdots\partial^{\alpha_n}_{x_n}$ and
$|\alpha|=|\alpha_1|+\cdots+|\alpha_n|$. For an integer $m$, the symbol $D^m_x$ denotes
the summation of all terms $\partial_x^\alpha$ with the multi-index $\alpha$ satisfying $|\alpha|=m$. We use $C_i$,
$\delta_i$, $K_i$, and $K$ to denote the constants which are independent of
$\epsilon$ and may change from line to line. We also omit the  spatial domain $\Omega $
in integrals for convenience.

(2) We shall frequently use the following Moser-type calculus
inequalities (see \cite{KM1}):

\hskip 4mm (i)\ \ For $f,g\in H^s(\Omega )\cap L^\infty(\Omega )$ and $|\alpha|\leq
s$, $s>n/2$, it holds that
\begin{align}\label{ma}
\|\partial^\alpha_x(fg)\| \leq C_s(\|f\|_{L^\infty}\|D^s_x
g\| +\|g\|_{L^\infty}\|D^s_x f\|).
\end{align}

\hskip 4mm (ii)\ \ For $f\in H^s(\Omega ), D_x^1 f\in L^\infty(\Omega ), g\in H^{s-1}(\Omega )\cap
L^\infty(\Omega )$ and $|\alpha|\leq s$, $s>n/2+1$, it holds that
\begin{align}\label{mb}
\quad  \ \ \|\partial^\alpha_x(fg)-f \partial^\alpha_xg\|\leq
C_s(\|D^1_x f\|_{L^\infty}\|D^{s-1}_x g\| +\|g\|_{L^\infty}\|D^s_xf\|). 
\end{align}

(3) Let $s> n/2$, $f\in C^s(\Omega )$, and  $u\in H^s(\Omega )$, then for each multi-index $\alpha$, $1\leq |\alpha| \leq s$, we have
(\cite{Mo,KM1}):
\begin{align}\label{mo}
   \|\partial^\alpha_x (f(u))\| \leq C(1+\|u\|_{L^\infty}^{|\alpha|-1})\|u\|_{|\alpha|}.
   \end{align}
Moreover, if $f(0)=0$, then (\cite{Ho97})
\begin{align}\label{ho}
 \|\partial^\alpha_x (f(u))\| \leq C( \|u\|_s)\|u\|_s.
\end{align}

This paper is organized as follows.
  In Section \ref{S2}, we utilize the primitive system \eqref{radya}--\eqref{radye} and the
target system \eqref{radza}--\eqref{radzd} to  derive the  error
system and state the local existence of the solution.
  In Section \ref{S3} we give the a priori energy
estimates of the error system  and present  the proof of Theorem \ref{th}.

\section{Derivation of the  error system} \label{S2}

In this section  we first derive the error system from the original
system \eqref{radya}--\eqref{radye} and the limiting equations
\eqref{radza}--\eqref{radzd}, then we state the local existence of solution to
this error system.

Setting $N^\epsilon=\rho^\epsilon-\rho^0,   \U^\epsilon=\u^\epsilon-\u^0, \Theta^\epsilon=\theta^\epsilon-\theta^0,
\F=I_0^\epsilon-(-\Delta)^{-1}\dv \mathbf{q}^0$, and $ \G = \mathbf{I}^{\epsilon}_1-\mathbf{q}^0$, and utilizing
      the system
\eqref{radya}--\eqref{radye} and the system \eqref{radza}--\eqref{radzd}, we
  obtain that
 \begin{align}
 &  \partial_t N^\epsilon  +(N^\epsilon+\rho^0)\dv \U^\epsilon+(\U^\epsilon+\u^0)\cdot\nabla N^\epsilon
 =-N^\epsilon \dv \u^0-\nabla \rho^0\cdot \U^\epsilon, \label{raderror1} \\
 & \partial_t\U^\epsilon  +[(\U^\epsilon+\u^0)\cdot \nabla]\U^\epsilon
 +\nabla \Theta^\epsilon+ \frac{\Theta^{\epsilon}+\theta^0}{N^\epsilon+\rho^0}\nabla N^\epsilon -\frac{1}{N^\epsilon+\rho^0}\dv\Psi(\U^\epsilon)\nonumber\\
  & \quad\qquad   =-(\U^\epsilon \cdot \nabla )\u^0-
                  \left[\frac{\Theta^{\epsilon}+\theta^0}{N^\epsilon+\rho^0}-\frac{\theta^0}{\rho^0}\right]\nabla  \rho^0
                       +\left[\frac{1}{N^\epsilon+\rho^0}-\frac{1}{\rho^0}\right]\dv\Psi(\u^0)\nonumber\\
 &   \qquad \qquad          +\frac{\epsilon}{N^\epsilon+\rho^0}(  \G +  \mathbf{q}^0),    \label{raderror2}\\
 & \partial_t\Theta^\epsilon  +[(\U^\epsilon+\u^0)\cdot \nabla]\Theta^\epsilon
 + (\Theta^\epsilon+\theta^0)\, \dv \U^\epsilon-\frac{\kappa}{N^\epsilon+\rho^0}\Delta \Theta^\epsilon\nonumber\\
  & \quad\qquad   =-(\U^\epsilon \cdot \nabla )\theta^0  -\Theta^\epsilon \dv \u^0
    +\left[\frac{\kappa}{N^\epsilon+\rho^0}-\frac{\kappa}{\rho^0}\right]\Delta \theta^0    \nonumber\\
    &\qquad \qquad + \frac{ 2\mu}{N^\epsilon+\rho^0} |\mathbb{D}(\U^\epsilon)|^2
    +\frac{\lambda}{N^\epsilon+\rho^0}|\mbox{tr}\mathbb{D}(\U^\epsilon)|^2\nonumber\\
      &\qquad \qquad
   + \frac{ 4\mu}{N^\epsilon+\rho^0}\mathbb{D}(\U^\epsilon): \mathbb{D}(\u^0)
     + \frac{ 2\lambda}{N^\epsilon+\rho^0}\,[\mbox{tr}\mathbb{D}(\U^\epsilon) \mbox{tr}\mathbb{D}(\u^0)]\nonumber\\
      &\qquad \qquad + \left[\frac{2\mu}{N^\epsilon+\rho^0}-\frac{2\mu}{\rho^0}\right]|\mathbb{D}(\u^0)|^2
      +\left[\frac{\lambda}{N^\epsilon+\rho^0}-\frac{\lambda}{\rho^0}\right](\mbox{tr}\mathbb{D}(\u^0))^2\nonumber\\
    &\qquad \qquad +\frac{1}{N^\epsilon+\rho^0}\big\{\F-(\Theta^\epsilon)^4-4(\Theta^\epsilon)^3\theta^0
    -6(\Theta^\epsilon)^2(\theta^0)^2-4\Theta^\epsilon(\theta^0)^3\big\} \nonumber\\
        &\qquad \qquad  -\left[\frac{1}{N^\epsilon+\rho^0}-\frac{1}{\rho^0}\right] \dv \mathbf{q}^0,      \label{raderror22}\\
  & \epsilon \partial_t \F+\dv \G
 = (\Theta^\epsilon)^4+4(\Theta^\epsilon)^3\theta^0
    +6(\Theta^\epsilon)^2(\theta^0)^2+4\Theta^\epsilon(\theta^0)^3\nonumber\\
     &\qquad \qquad\qquad \quad \ \   -  \F -\epsilon \partial_t(-\Delta)^{-1}\dv \mathbf{q}^0, \label{raderror3}\\
 & \epsilon\partial_t \G +\nabla \F=-\G-\epsilon\partial_t\mathbf{q}^0,  \label{raderror4}
 \end{align}
 with initial data
 \begin{align}\label{raderror5}
( N^\epsilon,\U^\epsilon,\Theta^\epsilon,\F,\G )|_{t=0} &\!=
 ( N^\epsilon_0,\U^\epsilon_0,\Theta^\epsilon_0,J^\epsilon_{00},\mathbf{J}^\epsilon_{10})\nonumber\\
  &\!= \big(\rho^\epsilon_0-\rho^0_0, \u^\epsilon_0-\u^0_0, \theta^\epsilon_0-\theta^0_0,
   I^\epsilon_0- (-\Delta)^{-1} \dv\mathbf{q}^0_0, \mathbf{I}^\epsilon_{10}-\mathbf{q}^0_0\big).
 \end{align}
We remark that in \eqref{raderror4} we have used the fact that
\begin{align*}
   \mathbf{q}^0=\nabla \dv \mathbf{q}^0-\nabla (\theta^0)^4=\Big(\frac{-\Delta}{I-\Delta}-I\Big)\nabla (\theta^0)^4
   =\frac{-\nabla}{I-\Delta}(\theta^0)^4
\end{align*}
is a gradient.

 Denote
 \begin{align*}
 &\W^\epsilon=\left(\begin{array}{c}
                   N^\epsilon \\
                    \U^\epsilon\\
                    \Theta^\epsilon\\
                   \F\\
                    \G
                  \end{array}\right),
                  \ \
       \W^\epsilon_0=\left(\begin{array}{c}
                     N^\epsilon_0 \\
                    \U^\epsilon_0\\
                     \Theta^\epsilon_0\\
                     J^\epsilon_{00}\\
                     \mathbf{J}^\epsilon_{00}\\
                  \end{array}\right), \ \
   \D^\epsilon=\left(\begin{array}{cc}
                   \mathbb{I}_{n+2} & \mathbf{0} \\
                   \mathbf{0} & \left(\begin{array}{cc}
                    \epsilon   & \mathbf{0}\\
                    \mathbf{0} &  \epsilon\mathbb{I}_{n}
                    \end{array}
                    \right)
                  \end{array}\right), \\
 &   \A^\epsilon_i=\left(\begin{array}{cc}
                  \left(\begin{array}{ccc}
                   (\U^\epsilon+\u^0)_i & (N^\epsilon+\rho^0) e^\mathrm{T}_i & 0  \\
                    \frac{\Theta^\epsilon+\theta^0}{N^\epsilon+\rho^0}e_i & (\U^\epsilon+\u^0)_i \mathbb{I}_{n} & e_i\\
                    0 & (\Theta^\epsilon+\theta^0)e^\mathrm{T}_i & (\U^\epsilon+\u^0)_i
                    \end{array}\right)   & \mathbf{0} \\
                   {0} &  \left(\begin{array}{cc}
                     {0}& e_{i}^\mathrm{T} \\
                    e_{i}  & \mathbf{0}
                    \end{array}
                    \right)
                  \end{array}\right), \\
  &         \A^\epsilon_{ij} =  \left(\begin{array}{cccc}
                                    {0} & \mathbf{0}& \mathbf{0}& \mathbf{0}\\
                              \mathbf{0}&   \frac{\mu}{N^\epsilon+\rho^0}( e_ie_j^\mathrm{T}\mathbb{I}_{n}+e_i^\mathrm{T}e_j)
                                +  \frac{\lambda}{N^\epsilon+\rho^0}e^{\mathrm{T}}_je_i &\mathbf{0}& \mathbf{0} \\
                                     {0} &   \mathbf{0}&\frac{\kappa}{N^\epsilon+\rho^0} e_ie_j^\mathrm{T}&   \mathbf{0}\\
                     \mathbf{0}  &   \mathbf{0} &   \mathbf{0}&  \mathbf{0}\\
                    \end{array}
                    \right),\\
  &\s^\epsilon(\W^\epsilon)=\left(\begin{array}{c}
                   -N^\epsilon \dv \u^0-\nabla \rho^0\cdot \U^\epsilon\\
                     {\R}^\epsilon_1\\
                    R^\epsilon_2\\
                    R^\epsilon_3\\
                  {-\G-\epsilon\partial_t\mathbf{q}^0}
                    \end{array}
                    \right),
 \end{align*}
where $\R^\epsilon_1, R^\epsilon_2$, and $ R^\epsilon_3$ denote the right-hand side of
\eqref{raderror2}, \eqref{raderror22}, and \eqref{raderror3}, respectively;
 $(e_1, \dots, e_n)$ is the canonical basis of $\mathbb{R}^n$
and $y_i$ denotes the $i$-th component of $y\in \mathbb{ R}^n$.

Using these notations we can rewrite the problem \eqref{raderror1}--\eqref{raderror5} in the form:
\begin{align}\label{raderror6}
  \left\{\begin{aligned}
&  \D^\epsilon \partial_t \W^\epsilon +\sum^{n}_{i=1}\A^\epsilon_i \partial_{x_i}\W^\epsilon
  +\sum^{n}_{i,j=1}\A^\epsilon_{ij}\partial_{x_ix_j}\W^\epsilon=\s^\epsilon(\W^\epsilon),\\
 & \W^\epsilon|_{t=0}= \W^\epsilon_0.
 \end{aligned} \right.
\end{align}
It is not difficult to see that the system for $\W^\epsilon$  in
\eqref{raderror6} can be reduced to  a quasilinear symmetric
hyperbolic-parabolic one. In fact, if we introduce
\begin{align}
\label{FXFXFX}
\A^\epsilon=\left(\begin{array}{cc}
                  \left(\begin{array}{ccc}
                   \frac{\Theta^\epsilon+\theta^0}{(N^\epsilon+\rho^0)^2} & \mathbf{0} & 0  \\
                   \mathbf{0} &  \mathbb{I}_{n} & \mathbf{0}\\
                   {0} & \mathbf{0} & \frac{1}{\Theta^\epsilon+\theta^0}\\
                                       \end{array}\right)   & \mathbf{0} \\
                    \mathbf{0} &   \mathbb{I}_{n+1}
                                   \end{array}\right),
\end{align}
which is positively definite when $\|N^\epsilon\|_{L^\infty_T L^\infty_x}\leq
\hat \rho/2$ and   $\|\Theta^\epsilon\|_{L^\infty_T L^\infty_x}\leq
\hat \theta/2$,
 then $\tilde{\A}_0^\epsilon=\A^\epsilon\D^\epsilon$
 is positive symmetric and $\tilde{\A}_i^\epsilon = \A^\epsilon\A^\epsilon_i$ are symmetric on $[0,T]$ for all $1\leq i\leq n.$
 Moreover, the assumptions that $\mu>0, 2\mu+n\lambda>0$, and $\kappa>0$ imply that
 $$
 \mathcal{A}^\epsilon=\sum^{n}_{i,j=1}\A^\epsilon\A^\epsilon_{ij}\partial_{x_ix_j}\W^\epsilon
 $$
 is an elliptic operator.
  Thus, for fixed $\epsilon>0$, we can apply the result  of Vol'pert and
Hudiaev~\cite{VH}    to obtain   the following local existence for the problem \eqref{raderror6}.

\begin{prop} \label{Pb}
Let $s>n/2+2 $ be an integer and $(\rho^0, \u^0, \theta^0,  \mathbf{q}^0)$ satisfy the conditions in Proposition \ref{Pa}.
 Assume that the initial data $(N^\epsilon_0, \U^\epsilon_0, \Theta^\epsilon_0, J_{00}^\epsilon, \mathbf{J}_{10}^\epsilon)$ satisfy
\begin{gather*}
  N^\epsilon_0, \U^\epsilon_0,\Theta^\epsilon_0,  J_{00}^\epsilon, \mathbf{J}_{10}^\epsilon\in
  H^s(\Omega )  \ \text{and} \ \ \|(N^\epsilon_0, \U^\epsilon_0, \Theta^\epsilon_0, J_{00}^\epsilon, \mathbf{J}_{10}^\epsilon)\|\leq \delta
\end{gather*}
  for some small constant $\delta>0$.
Then there exist positive constants $T^\epsilon\,(0<T^\epsilon\leq +\infty)$
 and $K$, such that the Cauchy problem
\eqref{raderror6} has a unique classical solution $(N^\epsilon,
\U^\epsilon, \Theta^\epsilon, J_0^\epsilon, \G )$ satisfying
\begin{gather*}
 N^\epsilon, J_0^\epsilon, \G  \in
C^l([0,T^\epsilon),H^{s-l}), \
\U^\epsilon, \Theta^\epsilon \in  C^l([0,T^\epsilon),H^{s-2l}),\  l=0,1; \\
  \|(N^\epsilon(t),U^\epsilon(t),\Theta^\epsilon(t),J_{0}^\epsilon(t),\mathbf{J}_{1}^\epsilon(t))\|_{s}\leq
   K\delta \ \ \ \ t\in [0,T^\epsilon).
  \end{gather*}
 \end{prop}

 Note that if  $\|N^\epsilon\|_{L^\infty_T L^\infty_x}\leq
\hat \rho/2$ and   $\|\Theta^\epsilon\|_{L^\infty_T L^\infty_x}\leq \hat \theta/2$,
then for smooth solutions, the   system \eqref{radya}--\eqref{radye} with initial data \eqref{radyf} are equivalent to
\eqref{raderror1}--\eqref{raderror5} or \eqref{raderror6} on $[0,T]$, $T<\min\{T^\epsilon,T_*\}$.
Usually, the life-span $T^\epsilon$ depends on $\epsilon$ and may shrink to zero as  as $\epsilon\to 0$.
 Therefore, in order to avoid this situation and to obtain the convergence of
  system \eqref{radya}--\eqref{radye} to the system \eqref{radza}--\eqref{radzd},
  we only need to establish the uniform decay estimates with respect
  to the parameter $\epsilon$ of the solution to the error system \eqref{raderror6}.
This will be achieved by the elaborate energy method presented in next section.

\section{Uniform energy estimates and proof of Theorem \ref{th}} \label{S3}
 In this section we derive the uniform decay estimates with respect to the parameter
  $\epsilon$ of the solution to the problem \eqref{raderror6} and
  justify rigorously the convergence of the system \eqref{radya}--\eqref{radye} to the  system \eqref{radza}--\eqref{radzd}.
Here we adopt and modify some techniques developed in \cite{JL,JL2} and put main efforts on the estimates
of higher order nonlinear terms.

We first establish the convergence rate of the error system by establishing the \emph{a priori} estimates
uniformly in $\epsilon$. For conciseness of presentation, we define
\begin{align*}
 &\|\mathcal{E}^\epsilon(t)\|^2_s\    := \|(N^\epsilon,\U^\epsilon,\Theta^\epsilon )(t)\|^2_{s},\\
 &\v \mathcal{E}^\epsilon(t)\v ^2_s :=\|\mathcal{E}^\epsilon(t)\| ^2_s+ {\epsilon }\Vert
(\F, \G)(t)\Vert^2_{s},\\
& \v\mathcal{E}^\epsilon\v_{s,T}\ : =\sup_{0<t<T}\v\mathcal{E}^\epsilon(t)\v_s.
\end{align*}

The crucial estimate of this paper is the following decay   result on the error system
\eqref{raderror1}--\eqref{raderror4}.

\begin{prop}\label{P31}
   Let $s>n/2+2$ be an integer and assume that the initial data  $( N^\epsilon_0,\U^\epsilon_0,\Theta^\epsilon_0, J^\epsilon_{00},\G _0)$ satisfy
\begin{align}\label{ww}
\|(N^\epsilon_0,\U^\epsilon_0,\Theta^\epsilon_0) \|^2_{s}+ \epsilon \Vert
( J_{00}^\epsilon, \mathbf{J}_{10}^\epsilon) \Vert^2_{s} =\v \mathcal{E}^\epsilon(t=0)\v^2_{s} \leq M_0{\epsilon}^2
 \end{align}
for sufficiently small $\epsilon$ and   some constant $M_0>0$   independent of $\epsilon$.
Then, for any $T_0\in (0, T_*)$, there exist
  two constants $ M_1 > 0$  and $\epsilon_1 > 0$  depending only on $T_0$, such that
for all $\epsilon\in (0,\epsilon_1]$, it holds that $T^\epsilon\geq T_0$
and the solution $( N^\epsilon, \U^\epsilon,\Theta^\epsilon, \F, \G )$ of the problem
\eqref{raderror1}--\eqref{raderror5}, well-defined in $[0, T_0]$, enjoys that
 \begin{align}\label{www}
   \v \mathcal{E}^\epsilon\v _{s,T_0} \leq M_1 {\epsilon}.
 \end{align}
   \end{prop}

Once this proposition is established, the proof of Theorem \ref{th} is a direct procedure. In fact, we have

\begin{proof}[Proof of Theorem \ref{th}]
Suppose that Proposition
\ref{P31} holds. According to the definition of the error functions $(N^\epsilon,\U^\epsilon,\Theta^\epsilon,\F,\G )$
and the regularity of $(\rho^0,\u^0,\theta^0,\mathbf{q}^0)$, the error system \eqref{raderror1}--\eqref{raderror4}  and the
primitive system \eqref{radya}--\eqref{radye} are equivalent on $[0,T]$ for some $T>0$.
Therefore  the assumption \eqref{ivda} in Theorem \ref{th} implies the assumption \eqref{ww} in Proposition
\ref{P31}, and hence  \eqref{www} gives \eqref{iivda}.
\end{proof}

Therefore, our main goal next is to prove Proposition \ref{P31} which can be approached by the following a priori
estimates. For some given  $\hat T<1$ and any $\tilde T<\hat T$ independent of $\epsilon$, we denote
$T \equiv  T_\epsilon = \min\{\tilde T, T^\epsilon \}$.

\begin{lem} Let the assumptions in Proposition \ref{P31} hold. Then, for all $0<t<T $ and sufficiently small $\epsilon$,
there exist two positive constants $\delta_1$ and $\delta_2$, such that
   \begin{align}\label{H2}
     &\v \mathcal{E}^\epsilon(t)\v ^2_s    +   \int^t_0\left\{\delta_1\|  \nabla \U^\epsilon\| ^2_{s}  +\delta_2\| \nabla \Theta^\epsilon\| ^2_{s}
+\frac{1}{4}\| \F\| ^2_{s}+\frac{1}{4}\| \G\| ^2_{s}\right\}(\tau)\dif\tau\nonumber\\
     \leq  &  \v \mathcal{E}^\epsilon(t=0)\v^2 _{s} + C\int^t_0\left
     \{(\| \mathcal{E}^\epsilon\| ^{2(s+1)}_{s}+1)\| \mathcal{E}^\epsilon\| ^2_{s}\right\}(\tau)\dif \tau
     +C\epsilon^2.
          \end{align}
\end{lem}
\begin{proof}
   Let $0\leq|\alpha|\leq s$.  In the following arguments the commutators will disappear in the case of $|\alpha|=0$.

 Applying the operator  $\partial^\alpha_x$ to \eqref{raderror1},
   multiplying  the resulting equation
    by $ \partial^\alpha_x N^\epsilon,$ and integrating over $\Omega $, we obtain that
    \begin{align}\label{hn1}
  \frac12\frac{\rm d}{{\rm d}t}\left\langle  \pa N^\epsilon, \pa N^\epsilon \right\rangle
 = & -\left\langle \pa([(\U^\epsilon+\u^0)\cdot\nabla] N^\epsilon),\pa N^\epsilon\right\rangle\nonumber\\
   &-\left\langle \pa((N^\epsilon+\rho^0)\dv \U^\epsilon), \pa N^\epsilon\right\rangle\nonumber\\
  & +\left\langle \pa((-N^\epsilon \dv \u^0-\nabla \rho^0\cdot \U^\epsilon),\pa N^\epsilon\right\rangle.
    \end{align}

The terms on the right-hand side of \eqref{hn1} can be bounded in a similar way to those in \cite{JL2}.
Here we present them for completeness. By the regularity of $\u^0$,
Cauchy-Schwarz's inequality, and Sobolev's imbedding theorem, we have
\begin{align}\label{hn2}
 & \ \ \ \ \langle \partial^\alpha_x([(\U^\epsilon+\u^0)\cdot \nabla]N^\epsilon),\partial^\alpha_x N^\epsilon\rangle\nonumber\\
& = \langle [(\U^\epsilon+\u^0)\cdot \nabla]\partial^\alpha_x N^\epsilon , \partial^\alpha_x N^\epsilon\rangle\
  +\big\langle \mathcal{H}^{(1)},\partial^\alpha_x N^\epsilon \big\rangle\nonumber\\
  & = -\frac12 \langle \dv(\U^\epsilon+\u^0) \partial^\alpha_x N^\epsilon , \partial^\alpha_x N^\epsilon \rangle\
   +\big\langle \mathcal{H}^{(1)},\partial^\alpha_x N^\epsilon \big \rangle\nonumber\\
 & \leq C(\| \mathcal{E}^\epsilon(t)\| _{s}+1)\| \partial^\alpha_x N^\epsilon\| ^2   +\|  \mathcal{H}^{(1)}\| ^2,
\end{align}
where the commutator
\begin{align*}
 \mathcal{H}^{(1)}  =\partial^\alpha_x([(\U^\epsilon+\u^0)\cdot \nabla]N^\epsilon)-[(\U^\epsilon+\u^0)\cdot \nabla]\partial^\alpha_x N^\epsilon
\end{align*}
can be estimated as follows. Using the Moser-type and Cauchy-Schwarz's inequalities,
the regularity of $\u^0$ and Sobolev's imbedding inequalities, we obtain that
\begin{align}\label{ca}
  \big\|\mathcal{H}^{(1)} \big\| &\leq C( \| D_x^1(\U^\epsilon+\u^0)\| _{L^\infty}\| D_x^sN^\epsilon\|
+\| D_x^1N^\epsilon\| _{L^\infty}\| D^{{s}}_x(\U^\epsilon+\u^0)\|) \nonumber\\
   & \leq C\| \mathcal{E}^\epsilon(t)\| _{s}^2+ C\| \mathcal{E}^\epsilon(t)\| _{s}.
\end{align}

Similarly, the second term on the right-hand side of \eqref{hn1} can bounded as follows.
\begin{align}\label{hn3}
  & \left\langle \pa((N^\epsilon+\rho^0)\dv \U^\epsilon), \pa N^\epsilon\right\rangle\nonumber\\
 & = \langle (N^\epsilon+\rho^0) \partial^\alpha_x\dv \U^\epsilon , \partial^\alpha_x N^\epsilon\rangle
  +\big\langle \mathcal{H}^{(2)},\partial^\alpha_x N^\epsilon \big\rangle\nonumber\\
 & \leq \eta_1 \|\nabla \partial^\alpha_x \U^\epsilon\|^2+ C_{\eta_1}(\| \mathcal{E}^\epsilon(t)\|^2 _{s}+1) \| \partial^\alpha_x N^\epsilon\| ^2
 +\big\|\mathcal{H}^{(2)}\big\| ^2
\end{align}
for any $\eta_1>0$, where the commutator
\begin{align*}
 \mathcal{H}^{(2)}  =\pa((N^\epsilon+\rho^0)\dv \U^\epsilon) - (N^\epsilon+\rho^0) \partial^\alpha_x\dv \U^\epsilon
\end{align*}
can be estimated by
\begin{align}\label{cb}
\big\|\mathcal{H}^{(2)} \big\| &\leq C( \| D_x^1(N^\epsilon+\rho^0)\| _{L^\infty}\| D_x^s\U^\epsilon\|
+\| D_x^1\U^\epsilon\| _{L^\infty}\| D^{s}_x(N^\epsilon+\rho^0)\| )\nonumber\\
   & \leq C\| \mathcal{E}^\epsilon(t)\| _{s}^2+ C\| \mathcal{E}^\epsilon(t)\| _{s}.
\end{align}

By the Moser-type and Cauchy-Schwarz's inequalities, and the regularity of $\u^0$ and $\rho^0$, we
can control the third term on the right-hand side of \eqref{hn1} by
\begin{align}\label{hn4}
\left|\left\langle \pa((-N^\epsilon \dv \u^0-\nabla \rho^0\cdot \U^\epsilon),\pa N^\epsilon\right\rangle\right|
\leq C\| \mathcal{E}^\epsilon(t)\| _{s}^2. 
\end{align}
 Substituting \eqref{hn2}--\eqref{hn4} into \eqref{hn1}, we conclude that
 \begin{align}\label{HN1}
  \frac12\frac{\rm d}{{\rm d}t}\left\langle  \pa N^\epsilon, \pa N^\epsilon \right\rangle
 \leq &  \eta_1 \|\nabla \partial^\alpha_x \U^\epsilon\|^2 \nonumber\\
 & + C\big[(\| \mathcal{E}^\epsilon(t)\|^2_{s}+1)\| \partial^\alpha_x N^\epsilon\| ^2+  (\| \mathcal{E}^\epsilon(t)\| _{s}^2+\| \mathcal{E}^\epsilon(t)\| _{s} )^2\big].
  \end{align}

   Applying the operator $\partial^\alpha_x$ to \eqref{raderror2}, multiplying  the resulting equations
    by $ \partial^\alpha_x\U^\epsilon$, and integrating over $\Omega $, we obtain that
    \begin{align}\label{HU2}
  & \frac12 \frac{\dif}{\dif  t}\langle \partial^\alpha_x\U^\epsilon, \partial^\alpha_x\U^\epsilon \rangle
  + \langle \partial^\alpha_x([(\U^\epsilon+\u^0)\cdot \nabla]\U^\epsilon),\partial^\alpha_x\U^\epsilon\rangle\nonumber\\
   &   +\left\langle \partial^\alpha_x\nabla \Theta^\epsilon, \partial^\alpha_x\U^\epsilon\right\rangle
  +  \left\langle \partial^\alpha_x\left( \frac{{\Theta^{\epsilon}+\theta^0}}{N^\epsilon+\rho^0}\nabla N^\epsilon\right), \partial^\alpha_x\U^\epsilon\right\rangle\nonumber\\
 & -  \left\langle   \partial^\alpha_x\left( \frac{1}{N^\epsilon+\rho^0} \dv\Psi(\U^\epsilon)\right) , \partial^\alpha_x \U^\epsilon\right\rangle\nonumber\\
   =     &    -  \left \langle\partial^\alpha_x\left[(\U^\epsilon \cdot \nabla )\u^0\right],   \partial^\alpha_x \U^\epsilon\right\rangle\nonumber\\
   &    - \left \langle\partial^\alpha_x\left\{\left[\frac{ {\Theta^{\epsilon}+\theta^0}}{N^\epsilon+\rho^0}-\frac{ {\theta^0}}{\rho^0}\right] \nabla \rho^0\right\},
   \partial^\alpha_x \U^\epsilon\right\rangle\nonumber\\
   &  +   \left \langle\partial^\alpha_x\left\{\left[\frac{1}{N^\epsilon+\rho^0}-\frac{1}{\rho^0}\right]\dv\Psi(\u^0)\right\},
   \partial^\alpha_x \U^\epsilon\right\rangle\nonumber\\
    & + \left\langle \partial^\alpha_x\left\{\frac{\epsilon}{N^\epsilon+\rho^0} (\G +\mathbf{q}^0)\right\} , \partial^\alpha_x\U^\epsilon \right\rangle \nonumber\\
 := &  \sum^{4}_{i=1}\mathcal{R}^{(i)}.
  \end{align}

The last four terms on the left-hand side of \eqref{HU2} can be estimated as follows. Similar to \eqref{hn2}, we infer that
\begin{align}\label{hu1}
  & \ \ \ \ \langle \partial^\alpha_x([(\U^\epsilon+\u^0)\cdot \nabla]\U^\epsilon),\partial^\alpha_x\U^\epsilon\rangle \nonumber\\
  & = \langle [(\U^\epsilon+\u^0)\cdot \nabla]\partial^\alpha_x \U^\epsilon , \partial^\alpha_x \U^\epsilon\rangle\
  +\big\langle \mathcal{H}^{(3)},\partial^\alpha_x \U^\epsilon \big\rangle\nonumber\\
  & = -\frac12 \langle \dv(\U^\epsilon+\u^0) \partial^\alpha_x \U^\epsilon , \partial^\alpha_x \U^\epsilon \rangle\
   +\big\langle \mathcal{H}^{(3)},\partial^\alpha_x \U^\epsilon \big \rangle\nonumber\\
 & \leq C(\| \mathcal{E}^\epsilon(t)\| _{s}+1)\| \partial^\alpha_x \U^\epsilon\| ^2   +\big\|  \mathcal{H}^{(3)}\big\| ^2,
\end{align}
where the commutator
\begin{align*}
 \mathcal{H}^{(3)}  =\partial^\alpha_x([(\U^\epsilon+\u^0)\cdot \nabla]\U^\epsilon)-[(\U^\epsilon+\u^0)\cdot \nabla]\partial^\alpha_x \U^\epsilon
\end{align*}
can be bounded by
\begin{align}\label{hu2}
\big\|\mathcal{H}^{(3)} \big\| &\leq C( \| D_x^1(\U^\epsilon+\u^0)\| _{L^\infty}\| D_x^s\U^\epsilon\|
+\| D_x^1\U^\epsilon\| _{L^\infty}\| D^{s}_x(\U^\epsilon+\u^0)\|) \nonumber\\
   & \leq C\| \mathcal{E}^\epsilon(t)\| _{s}^2+ {C\| \mathcal{E}^\epsilon(t)\| _{s}}.
\end{align}

By H\"older's inequality, we have
\begin{align}\label{hu3}
 \left\langle \partial^\alpha_x\nabla \Theta^\epsilon, \partial^\alpha_x\U^\epsilon\right\rangle
\leq \eta_2    \|\partial^\alpha_x\nabla \Theta^\epsilon\|^2+C_{\eta_2}\|\partial^\alpha_x\U^\epsilon\|^2
\end{align}
for any $\eta_2>0$.
For the fourth term on the left-hand side of \eqref{HU2}, similar to \eqref{hn3}, we integrate it by parts to deduce that
\begin{align}\label{hu4}
 & \quad\,  \left\langle \partial^\alpha_x\left( \frac{{\Theta^{\epsilon}+\theta^0}}{N^\epsilon+\rho^0}\nabla N^\epsilon\right), \partial^\alpha_x\U^\epsilon\right\rangle \nonumber\\
& = \left\langle  \frac{\Theta^{\epsilon}+\theta^0}{N^\epsilon+\rho^0}\partial^\alpha_x \nabla N^\epsilon , \partial^\alpha_x\U^\epsilon\right\rangle
 + \big\langle \mathcal{H}^{(4)}, \partial^\alpha_x\U^\epsilon\big\rangle\nonumber\\
&   =  -\left\langle\partial^\alpha_x   N^\epsilon , \dv\left(\frac{\Theta^{\epsilon}+\theta^0}{N^\epsilon+\rho^0}\partial^\alpha_x\U^\epsilon\right)  \right\rangle\
  +\big\langle \mathcal{H}^{(4)},\partial^\alpha_x \U^\epsilon \big\rangle\nonumber\\
&   \leq   \eta_3 \|\nabla \partial^\alpha_x \U^\epsilon\|^2+ C_{\eta_3} \| \partial^\alpha_x N^\epsilon\| ^2+C(\| \mathcal{E}^\epsilon(t)\| _{s}+1)\|\partial^\alpha_x\U^\epsilon\|^2
 +\big\|  \mathcal{H}^{(4)}\big\| ^2
\end{align}
for any $\eta_3>0$,
where the commutator
\begin{align*}
 \mathcal{H}^{(4)}  =
  \partial^\alpha_x\left( \frac{\Theta^{\epsilon}+\theta^0}{N^\epsilon+\rho^0}\nabla N^\epsilon\right)
  -   \frac{\Theta^{\epsilon}+\theta^0}{N^\epsilon+\rho^0}\partial^\alpha_x \nabla N^\epsilon
\end{align*}
can be bounded as follows, using \eqref{mb}, \eqref{mo} and Cauchy-Schwarz's inequality:
 \begin{align}\label{hu44}
\big\|\mathcal{H}^{(4)}\big\| &\leq C\left( \left\| D_x^1\left( \frac{\Theta^{\epsilon}+\theta^0}{N^\epsilon+\rho^0}\right)\right\|_{L^\infty}\| D_x^sN^\epsilon\|
+\| D_x^1N^\epsilon\| _{L^\infty}\left\| D^{s}_x\left( \frac{\Theta^{\epsilon}+\theta^0}{N^\epsilon+\rho^0}\right)\right\|\right) \nonumber\\
   & \leq C(\| \mathcal{E}^\epsilon(t)\| _{s}^2+ \| \mathcal{E}^\epsilon(t)\| _{s}^{2(s+1)} +\| \mathcal{E}^\epsilon(t)\| _{s}).
\end{align}
For the fifth term on the left-hand side of \eqref{HU2}, we have
\begin{align} \label{hu5}
   &-  \left\langle   \partial^\alpha_x\left( \frac{1}{N^\epsilon+\rho^0} \dv\Psi(\U^\epsilon)\right) , \partial^\alpha_x \U^\epsilon\right\rangle\nonumber\\
  =& -  \left\langle  \frac{1}{N^\epsilon+\rho^0}  \partial^\alpha_x \dv\Psi(\U^\epsilon)  , \partial^\alpha_x\U^\epsilon\right\rangle
   -  \big\langle\mathcal{H}^{(5)}, \partial^\alpha_x\U^\epsilon\big\rangle,
  \end{align}
where the commutator
\begin{align*}
   \mathcal{H}^{(5)}=  \partial^\alpha_x\left( \frac{1}{N^\epsilon+\rho^0} \dv\Psi(\U^\epsilon)\right)
  - \frac{1}{N^\epsilon+\rho^0}  \partial^\alpha_x \dv\Psi(\U^\epsilon).
\end{align*}
By the Moser-type and Cauchy-Schwarz inequalities, the regularity of $\rho^0$ and the
positivity of $N^\epsilon+\rho_0$,  the definition of  $\Psi(\U^\epsilon)$ and Sobolev's imbedding theorem, we find that
 \begin{align}\label{hu6}
 & \ \ \   \big|\big\langle\mathcal{H}^{(5)},
 \partial^\alpha_x\U^\epsilon\big\rangle\big|
    \leq  \big\|\mathcal{H}^{(5)}\big\|\cdot \| \partial^\alpha_x\U^\epsilon \|  \nonumber\\
   & \leq C \left(\left\Vert D_x^1\left(\frac{1}{N^\epsilon+\rho^0}\right)\right\Vert_{L^\infty}\|  \dv\Psi(\U^\epsilon)\| _{s-1}
   +\| \dv\Psi(\U^\epsilon)\| _{L^\infty}\left\Vert\frac{1}{N^\epsilon+\rho^0}\right\Vert_{s}\right)\| \partial^\alpha_x\U^\epsilon \| \nonumber\\
   &\leq \eta_4 \| \nabla \U^\epsilon\| ^2_{s}+C_{\eta_4} (\| \mathcal{E}^\epsilon(t)\| ^2_{s}+1)\| \partial^\alpha_x
   \U^\epsilon\| ^2+C\| \mathcal{E}^\epsilon(t)\|^2_{s}(1+\| \mathcal{E}^\epsilon(t)\|_{s}^{2s})
 \end{align}
 for any $\eta_4>0$, where we have used the assumption $s>n/2+2$
and the imbedding $H^l(\Omega )\hookrightarrow L^\infty (\mathbb{R}^n)$ for $l>n/2$.
By virtue of the definition of $\Psi(\U^\epsilon)$ and integration by parts, the first term on the right-hand side of \eqref{hu5}
can be rewritten as
    \begin{align}\label{hu7}
 &  -  \left\langle  \frac{1}{N^\epsilon+\rho^0}  \partial^\alpha_x \dv\Psi(\U^\epsilon)  , \partial^\alpha_x\U^\epsilon\right\rangle\nonumber\\
 =&2\mu \left\langle \frac{1}{N^\epsilon+\rho^0}\partial^\alpha_x\mathbb{D}(\U^\epsilon),\partial^\alpha_x\mathbb{D}(\U^\epsilon)\right\rangle
+\lambda \left\langle \frac{1}{N^\epsilon+\rho^0}\partial^\alpha_x\dv \U^\epsilon ,\partial^\alpha_x \dv \U^\epsilon \right\rangle
\nonumber\\
& +2\mu \left\langle \nabla\left(\frac{1}{N^\epsilon+\rho^0}\right)\otimes \partial^\alpha_x\U^\epsilon,\partial^\alpha_x\mathbb{D}(\U^\epsilon)\right\rangle\nonumber\\
&+\lambda \left\langle \nabla\left(\frac{1}{N^\epsilon+\rho^0}\right)\cdot\partial^\alpha_x \U^\epsilon ,\partial^\alpha_x \dv \U^\epsilon \right\rangle
\nonumber\\
:=& \sum^4_{i=1}\mathcal{I}^{(i)}.
\end{align}
Recalling that $\mu>0$ and $2\mu+n\lambda>0$, and the positivity of $N^\epsilon+\rho_0$,
the first two terms $\mathcal{I}^{(1)}$ and $\mathcal{I}^{(2)}$ can be bounded as follows.
\begin{align}\label{hu8}
 \mathcal{I}^{(1)}+\mathcal{I}^{(2)} &=\int  \frac{1}{N^\epsilon+\rho^0}\big\{2\mu|\pa\mathbb{D}
 (\U^\epsilon)|^2+\lambda|\pa\mbox{tr}\mathbb{D}(\U^\epsilon)|^2\big\}{\rm d}x\nonumber\\
&  \geq 2 \mu \int \frac{1}{N^\epsilon+\rho^0}\left(|\pa\mathbb{D}(\U^\epsilon)|^2-\frac1n |\pa\mbox{tr}\mathbb{D}(\U^\epsilon)|^2\right){\rm d}x\nonumber\\
  & =   \mu \int \frac{1}{N^\epsilon+\rho^0}\left(|\pa\nabla\U^\epsilon|^2+\frac1n |\pa\dv  \U^\epsilon|^2\right){\rm d}x\nonumber\\
  & \geq   \mu \int \frac{1}{N^\epsilon+\rho^0}|\pa\nabla\U^\epsilon|^2{\rm d}x.
\end{align}

  By virtue of Cauchy-Schwarz's inequality, the regularity of $\rho^0$ and the
positivity of $N^\epsilon+\rho_0$, the terms $\mathcal{I}^{(3)}$ and $\mathcal{I}^{(4)}$ can be controlled by
 \begin{align}\label{hu9}
|\mathcal{I}^{(3)}|+|\mathcal{I}^{(4)}|
    &\leq \eta_5 \|\nabla \pa\U^\epsilon\|^2 +C_{\eta_5} (\|\mathcal{E}^\epsilon(t)\|^2_{s}+1)(\|\partial^\alpha_x
   \U^\epsilon\|^2 +\|\partial^\alpha_x N^\epsilon\|^2)
 \end{align}
 for any $\eta_5>0$, where the assumption $s>n/2+2$ has been used.

Substituting \eqref{hu1}--\eqref{hu9} into \eqref{HU2}, we conclude that
\begin{align}\label{HU20}
&   \frac12\frac{\dif}{\dif t}\langle \partial^\alpha_x\U^\epsilon,
\partial^\alpha_x\U^\epsilon \rangle
  +  \int \frac{\mu}{N^\epsilon+\rho^0} |\nabla \partial^\alpha_x \U^\epsilon|^2 \dif x-
  (\eta_3+\eta_4+\eta_5)\|\nabla \U^\epsilon\|_s^2\nonumber\\
  \leq &
   C_{{\eta}}\big\{ (\| \mathcal{E}^\epsilon(t)\| ^2_{s}+1)(\| \partial^\alpha_x
   \U^\epsilon\| ^2+\| \partial^\alpha_x  N^\epsilon\| ^2+\| \mathcal{E}^\epsilon(t)\|_{s}^{2s})+ \| \mathcal{E}^\epsilon(t)\|_{s}^{2} \big\}\nonumber\\
   & + \eta_2    \|\partial^\alpha_x\nabla \Theta^\epsilon\|^2 +\sum^{4}_{i=1}\mathcal{R}^{(i)}
\end{align}
for some constant $C_{{\eta}}>0$ depending on $\eta_i$ ($i=1,\dots,5$).

 We have to estimate the terms on the right-hand side of \eqref{HU20}.
  In view of the regularity of $(\rho^0,\u^0,\mathbf{q}^0)$,
the positivity of $ N^\epsilon+\rho^0$ and Cauchy-Schwarz's inequality, the first
two terms $\mathcal{R}^{(1)}$ and $\mathcal{R}^{(2)}$ can be controlled by
\begin{align}\label{hu10}
  \big|\mathcal{R}^{(1)}  \big|+  \big|\mathcal{R}^{(2)}  \big| \leq  C(\|\mathcal{E}^\epsilon(t)\|_{s}^2 +\|\mathcal{E}^\epsilon(t)\|_{s}^{2s})
  +C\|\partial^\alpha_x\U^\epsilon\|^2.
\end{align}

For the term $\mathcal{R}^{(3)}$,  by the regularity of $ \rho^0$ and $\u^0$,
the positivity of $ N^\epsilon+\rho^0$,  Cauchy-Schwarz's inequality and \eqref{ho}, we see that
\begin{align}\label{hu11}
     \big|\mathcal{R}^{(3)}  \big| \leq  C\|\mathcal{E}^\epsilon(t)\|_{s}^2+
   C\|\partial^\alpha_x\U^\epsilon\|^2.
\end{align}

For the last term $\mathcal{R}^{(4)}$,   we utilize the positivity of $N^\epsilon+\rho^0$ to deduce that
\begin{align}\label{hu12}
    \big| \mathcal{R}^{(4)}  \big|& =  \left|\left\langle\frac{\epsilon}{N^\epsilon+\rho^0}\partial^\alpha_x \G,
   \partial^\alpha_x\U^\epsilon \right\rangle\right|
+    \big|\big\langle \mathcal{H}^{(6)},
\partial^\alpha_x\U^\epsilon \big\rangle  \big| +   \big|\mathcal{R}^{(4_1)}  \big|\nonumber\\
 & \leq \frac{1}{16} \| \partial^\alpha_x \G\| ^2+
 C\| \partial^\alpha_x\U^\epsilon\| ^2+    \big|\big\langle \mathcal{H}^{(6)}, \partial^\alpha_x\U^\epsilon\big\rangle  \big| +    \big|\mathcal{R}^{(4_1)}  \big|,
\end{align}
where
\begin{align*}
 \mathcal{H}^{(6)}  =\epsilon\partial^\alpha_x\left\{
\frac{\G}{N^\epsilon+\rho^0} \right\}-\frac{\epsilon }{N^\epsilon+\rho^0}\partial^\alpha_x \G
\nonumber\
\end{align*}
and
\begin{align*}
\mathcal{R}^{(4_1)}=\epsilon\left\langle\partial^\alpha_x\left\{
\frac{\mathbf{q}^0}{N^\epsilon+\rho^0}\right\},
   \partial^\alpha_x\U^\epsilon \right\rangle.
\end{align*}

If we make use of the Moser-type inequality, \eqref{mo} and the regularity of $\rho^0$ and $\mathbf{q}^0$, we obtain that
 \begin{align}\label{hu13}
 &  \ \ \   \big|\big\langle\mathcal{H}^{(6)},
  \partial^\alpha_x\U^\epsilon\big\rangle\big| \leq  \big\|\mathcal{H}^{(6)}\big\|\cdot \| \partial^\alpha_x\U^\epsilon \|  \nonumber\\
   & \leq C \left[\left\Vert D_x^1\left(\frac{1}{N^\epsilon+\rho^0}\right)\right\Vert_{L^\infty}\| \G\| _{s-1}
   +\| \G\| _{L^\infty}\left\Vert\frac{1}{N^\epsilon+\rho^0}\right\Vert_{s}\right]\| \partial^\alpha_x\U^\epsilon \| \nonumber\\
   &\leq \eta_6 \| \G\| ^2_{s-1}+C_{\eta_6} (\| \mathcal{E}^\epsilon(t)\| ^{2(s+1)}_{s}+1) \| \partial^\alpha_x
   \U^\epsilon\| ^2
 \end{align}
 for any $\eta_6>0$.
 Recalling the regularity of $\u^0$ and $\mathbf{q}^0$, \eqref{ma} and \eqref{mo} and H\"{o}lder's
inequality, we find that
\begin{align}\label{hu14}
\big|\mathcal{R}^{(4_1)}\big|\leq  C(\| \mathcal{E}^\epsilon(t)\| _{s}^{2s}
+1)\| \partial^\alpha_x\U^\epsilon\| ^2 +C\epsilon^2.
\end{align}

Substituting  \eqref{hu10}--\eqref{hu14} into \eqref{HU20}, we conclude that
 \begin{align}\label{HU202}
&   \frac12\frac{\dif}{\dif t}\langle \partial^\alpha_x\U^\epsilon,
\partial^\alpha_x\U^\epsilon \rangle
  +  \int \frac{\mu}{N^\epsilon+\rho^0} |\nabla \partial^\alpha_x \U^\epsilon|^2 \dif x-
  (\eta_3+\eta_4+\eta_5)\|\nabla \U^\epsilon\|_s^2\nonumber\\
\leq   &
    \tilde{C}_{{\eta}}\big[(\| \mathcal{E}^\epsilon(t)\| _{s}^{2(s+1)}+1) \| \mathcal{E}^\epsilon(t)\| ^2_{s} \big]\nonumber\\
  & +\eta_2    \|\partial^\alpha_x\nabla \Theta^\epsilon\|^2+\left(\eta_6+\frac{1}{16}\right)\| \G\| ^2_{s}+C\epsilon^2
\end{align}
for some constant $\tilde{C}_{{\eta}}>0$ depending on $\eta_i$ ($i=2,\dots,6$).

  Applying the operator $\partial^\alpha_x$ to \eqref{raderror22}, multiplying  the resulting equation
    by $ \partial^\alpha_x\Theta^\epsilon$, and integrating over $\Omega $, we arrive at
\begin{align}  \label{HT1}
 & \frac12\frac{\dif}{\dif t}\langle\pa\Theta^\epsilon,\pa\Theta^\epsilon\rangle  +\langle \pa\{[(\U^\epsilon+\u^0)\cdot \nabla]\Theta^\epsilon\},\pa\Theta^\epsilon\rangle    \nonumber\\
& + \left\langle\pa\{(\Theta^\epsilon+\theta^0)\, \dv \U^\epsilon\},\pa\Theta^\epsilon\right\rangle-\left\langle\pa\left\{\frac{\kappa}{N^\epsilon+\rho^0}\Delta \Theta^\epsilon\right\},\pa\Theta^\epsilon\right\rangle\nonumber\\
 =  &-\langle \pa\{(\U^\epsilon \cdot \nabla )\theta^0
  -\Theta^\epsilon \dv \u^0\},\pa\Theta^\epsilon\rangle\nonumber\\
  & +\left\langle\pa\left\{\left[\frac{\kappa}{N^\epsilon+\rho^0}-\frac{\kappa}{\rho^0}\right]\Delta \theta^0 \right\},\pa\Theta^\epsilon\right\rangle   \nonumber\\
  & +\left\langle\pa\left\{\left[\frac{2\mu}{N^\epsilon+\rho^0}-\frac{2\mu}{\rho^0}\right]|\mathbb{D}(\u^0)|^2 \right\},\pa\Theta^\epsilon\right\rangle   \nonumber\\
  & +\left\langle\pa\left\{\left[\frac{\lambda}{N^\epsilon+\rho^0}-\frac{\lambda}{\rho^0}\right](\mbox{tr}\mathbb{D}(\u^0))^2 \right\},\pa\Theta^\epsilon\right\rangle   \nonumber\\
  & +\left\langle\pa\left\{\left[\frac{1}{N^\epsilon+\rho^0}-\frac{1}{\rho^0}\right]\dv \mathbf{q}^0 \right\},\pa\Theta^\epsilon\right\rangle   \nonumber\\
  &  + \left\langle\pa\left\{\frac{ 2\mu}{N^\epsilon+\rho^0} |\mathbb{D}(\U^\epsilon)|^2+\frac{\lambda}{N^\epsilon+\rho^0}|\mbox{tr}\mathbb{D}(\U^\epsilon)|^2\right\},\pa\Theta^\epsilon\right\rangle\nonumber\\
  &   + \left\langle\pa\left\{\frac{ 4\mu}{N^\epsilon+\rho^0}\mathbb{D}(\U^\epsilon): \mathbb{D}(\u^0)\right\},\pa\Theta^\epsilon\right\rangle\nonumber\\
   &  + \left\langle\pa\left\{\frac{ 2\lambda}{N^\epsilon+\rho^0}\,[\mbox{tr}\mathbb{D}(\U^\epsilon) \mbox{tr}\mathbb{D}(\u^0)]\right\},\pa\Theta^\epsilon\right\rangle\nonumber\\
    & +\left\langle\pa\left\{ \frac{\F}{N^\epsilon+\rho^0}  \right\},\pa\Theta^\epsilon\right\rangle
    - \left\langle\pa\left\{ \frac{(\Theta^\epsilon)^4}{N^\epsilon+\rho^0}  \right\},\pa\Theta^\epsilon\right\rangle\nonumber\\
    & \nonumber\\
     & -\left\langle\pa\left\{ \frac{4(\Theta^\epsilon)^3\theta^0}{N^\epsilon+\rho^0}  \right\},\pa\Theta^\epsilon\right\rangle
    - \left\langle\pa\left\{ \frac{6(\Theta^\epsilon)^2(\theta^0)^2}{N^\epsilon+\rho^0}  \right\},\pa\Theta^\epsilon\right\rangle\nonumber\\
    & \nonumber\\
     & -\left\langle\pa\left\{ \frac{4\Theta^\epsilon(\theta^0)^3}{N^\epsilon+\rho^0}  \right\},\pa\Theta^\epsilon\right\rangle\nonumber\\
      := & \sum^{13}_{i=1}\mathcal{S}^{(i)}.
\end{align}

We first bound the terms on the left-hand side of \eqref{HT1}. Similar to \eqref{hn2}, we  have
\begin{align}\label{ht1}
 &  \ \ \ \ \langle \partial^\alpha_x([(\U^\epsilon+\u^0)\cdot \nabla]\Theta^\epsilon),\partial^\alpha_x\Theta^\epsilon\rangle\nonumber\\
  & = \langle [(\U^\epsilon+\u^0)\cdot \nabla]\partial^\alpha_x \Theta^\epsilon , \partial^\alpha_x \Theta^\epsilon\rangle\
  +\big\langle \mathcal{H}^{(7)},\partial^\alpha_x \Theta^\epsilon \big\rangle\nonumber\\
  & = -\frac12 \langle \dv(\U^\epsilon+\u^0) \partial^\alpha_x \Theta^\epsilon , \partial^\alpha_x \Theta^\epsilon \rangle\
   +\big\langle \mathcal{H}^{(7)},\partial^\alpha_x \Theta^\epsilon \big \rangle\nonumber\\
 & \leq C(\| \mathcal{E}^\epsilon(t)\| _{s}+1)\| \partial^\alpha_x \Theta^\epsilon\| ^2   +\big\|  \mathcal{H}^{(7)}\big\| ^2,
\end{align}
where the commutator
\begin{align*}
 \mathcal{H}^{(7)}  =\partial^\alpha_x([(\U^\epsilon+\u^0)\cdot \nabla]\Theta^\epsilon)-[(\U^\epsilon+\u^0)\cdot \nabla]\partial^\alpha_x \Theta^\epsilon
\end{align*}
can be bounded by
\begin{align}\label{ht2}
  \big\|\mathcal{H}^{(7)} \big\| &\leq C( \| D_x^1(\U^\epsilon+\u^0)\| _{L^\infty}\| D_x^s\Theta^\epsilon\|
+\| D_x^1\Theta^\epsilon\| _{L^\infty}\| D^{s}_x(\U^\epsilon+\u^0)\|) \nonumber\\
   & \leq C\| \mathcal{E}^\epsilon(t)\| _{s}^2+C\| \mathcal{E}^\epsilon(t)\| _{s}.
\end{align}
The second term on the left-hand side of \eqref{HT1} can bounded, similarly to \eqref{hn3}, as follows.
\begin{align}\label{ht3}
  &   \big|\left\langle \pa((\Theta^\epsilon+\theta^0)\dv \U^\epsilon), \pa \Theta^\epsilon\right\rangle  \big|\nonumber\\
 \leq &  \big|\langle [(\Theta^\epsilon+\rho^0) \partial^\alpha_x\dv \U^\epsilon , \partial^\alpha_x \Theta^\epsilon\rangle  \big|
  +  \big|\big\langle \mathcal{H}^{(8)},\partial^\alpha_x \Theta^\epsilon \big\rangle  \big|\nonumber\\
  \leq  & \eta_7 \|\nabla \partial^\alpha_x \U^\epsilon\|^2+ C_{\eta_7} \| \partial^\alpha_x \Theta^\epsilon\| ^2   +\big\|  \mathcal{H}^{(8)}\big\| ^2
  \qquad \mbox{for any }\eta_7>0,
\end{align}
where the commutator
\begin{align*}
 \mathcal{H}^{(8)}  =\pa((\Theta^\epsilon+\rho^0)\dv \U^\epsilon) - (\Theta^\epsilon+\theta^0) \partial^\alpha_x\dv \U^\epsilon
\end{align*}
can be controlled by
\begin{align}\label{ht4}
  \big\|\mathcal{H}^{(8)}\big\| &\leq C( \| D_x^1(\Theta^\epsilon+\theta^0)\| _{L^\infty}\| D_x^s\U^\epsilon\|
+\| D_x^1\U^\epsilon\| _{L^\infty}\| D^{s}_x(\Theta^\epsilon+\theta^0)\| )\nonumber\\
   & \leq C\| \mathcal{E}^\epsilon(t)\| _{s}^2+C\| \mathcal{E}^\epsilon(t)\| _{s}.
\end{align}

For the fourth term on the left-hand side of \eqref{HT1}, we integrate it by parts to deduce that
\begin{align} \label{ht5}
   &- \kappa \left\langle \partial^\alpha_x\left( \frac{1}{N^\epsilon+\rho^0}\Delta \U^\epsilon\right), \partial^\alpha_x\Theta^\epsilon\right\rangle
   \nonumber \\
  & = - \kappa  \left\langle  \frac{1}{N^\epsilon+\rho^0}\Delta \partial^\alpha_x \Theta^\epsilon , \partial^\alpha_x\Theta^\epsilon\right\rangle
   -  \kappa\big\langle\mathcal{H}^{(9)}, \partial^\alpha_x\Theta^\epsilon\big\rangle  \nonumber\\
 & = \kappa \left\langle  \frac{1}{N^\epsilon+\rho^0}\nabla\partial^\alpha_x \Theta^\epsilon , \nabla\partial^\alpha_x\Theta^\epsilon\right\rangle
     \nonumber\\
   & \quad + \kappa\left \langle \nabla \left(\frac{1}{N^\epsilon+\rho^0}\right)   \nabla\partial^\alpha_x \Theta^\epsilon,
\partial^\alpha_x \Theta^\epsilon\right\rangle -  \kappa\big\langle\mathcal{H}^{(9)}, \partial^\alpha_x\Theta^\epsilon\big\rangle,
   \end{align}
where
\begin{align*}
   \mathcal{H}^{(9)}=  \partial^\alpha_x\left( \frac{1}{N^\epsilon+\rho^0}\Delta \Theta^\epsilon\right)
   - \frac{1}{N^\epsilon+\rho^0}\Delta \partial^\alpha_x \Theta^\epsilon.
\end{align*}
By the Moser-type and H\"{o}lder inequalities, the regularity of $\rho^0$, the
positivity of $N^\epsilon+\rho_0$ and \eqref{mo}, we find that
 \begin{align}\label{ht6}
 & \ \ \   \big|\big\langle\mathcal{H}^{(9)},
 \partial^\alpha_x\Theta^\epsilon\big\rangle\big|
    \leq  \big\|\mathcal{H}^{(9)}\big\|\cdot \| \partial^\alpha_x\Theta^\epsilon \|  \nonumber\\
   & \leq C \left(\left\Vert D_x^1\left(\frac{1}{N^\epsilon+\rho^0}\right)\right\Vert_{L^\infty}\| \Delta \Theta^\epsilon\| _{s-1}
   +\| \Delta \Theta^\epsilon\| _{L^\infty}\left\Vert\frac{1}{N^\epsilon+\rho^0}\right\Vert_{s}\right)\| \partial^\alpha_x\Theta^\epsilon \| \nonumber\\
   &\leq \eta_8 \| \nabla \Theta^\epsilon\| ^2_{s}+C_{\eta_8} (\| \mathcal{E}^\epsilon(t)\| ^{2s}_{s}+1)\| \mathcal{E}^\epsilon(t)\|^2_s
 \end{align}
 and
 \begin{align}\label{ht7}
 & \left|\left \langle \nabla \left(\frac{1}{N^\epsilon+\rho^0}\right) \nabla{\partial^\alpha_x \Theta^\epsilon} ,
\partial^\alpha_x \Theta^\epsilon\right\rangle\right|\nonumber\\
   &  \leq\eta_{9}\|\nabla \partial^\alpha_x\Theta^\epsilon\|^2
   + C_{\eta_{9}} \left\Vert \nabla\left(\frac{1}{N^\epsilon+\rho^0}\right)\right\Vert_{L^\infty}^2
   \|\partial^\alpha_x   \Theta^\epsilon\|^2\nonumber\\
   &  \leq \eta_{9}\|\nabla \partial^\alpha_x\Theta^\epsilon\|^2
   + CC_{\eta_{9}}  (\|\mathcal{E}^\epsilon(t)\|^2_{s}+1)\| \partial^\alpha_x\Theta^\epsilon\|^2
   \end{align}
 for any $\eta_8>0 $ and $\eta_{9}>0$, where
 we have used the assumption  $s>n/2+2$ in the derivation of \eqref{ht6}
and the imbedding $H^l ( \Omega )\hookrightarrow L^\infty (\mathbb{R}^n)$ for $l>n/2$.

 Now, we estimate every term on the right-hand side of \eqref{HT1}.
 By virtue of the regularity of $ \theta^0$ and $\u^0$, and Cauchy-Schwarz's inequality, the first
 term $\mathcal{S}^{(1)}$ can be estimated as follows.
\begin{align}\label{ht8}
  \big| \mathcal{S}^{(1)} \big| \leq  C(\|\mathcal{E}^\epsilon(t)\|_{s}^2 +1)
   ( \|\partial^\alpha_x \Theta^\epsilon\|^2 +\|\partial^\alpha_x\U^\epsilon\|^2).
\end{align}

For the  terms $\mathcal{S}^{(i)}\,(i=2,3,4,5)$, we utilize the regularity of $ \rho^0$, $\u^0$ and $\mathbf{q}^0$,
the positivity of $ N^\epsilon+\rho^0$, Cauchy-Schwarz's inequality and \eqref{ho} to deduce that
\begin{align}\label{ht88}
     \big|\mathcal{S}^{(2)}  \big|+  \big|\mathcal{S}^{(3)}  \big| +  \big|\mathcal{S}^{(4)}  \big|+  \big|\mathcal{S}^{(5)}  \big|
   \leq  C(\|\mathcal{E}^\epsilon(t)\|_{s}^{2s} +\|\mathcal{E}^\epsilon(t)\|_{s}^{2})+
   C\|\partial^\alpha_x\Theta^\epsilon\|^2,
\end{align}
while for the sixth term $\mathcal{S}^{(6)}$, we integrate by parts, use Cauchy-Schwarz's inequality and the positivity of
$\Theta^\epsilon+\rho^0$ to obtain that
\begin{align}\label{ht9}
  \big|\mathcal{S}^{(6)}  \big|= & \left| \left\langle\partial_x^{\alpha-\alpha_1}\left\{\frac{ 2\mu}{N^\epsilon+\rho^0} |\mathbb{D}(\U^\epsilon)|^2
+\frac{\lambda}{N^\epsilon+\rho^0}|\mbox{tr}\mathbb{D}(\U^\epsilon)|^2\right\},\partial_x^{\alpha+\alpha_1}\Theta^\epsilon\right\rangle\right|\nonumber\\
 \leq &\eta_{10} \|\nabla \pa \Theta^\epsilon\|^2+C_{\eta_{10}}(\|\mathcal{E}^\epsilon(t)\|_{s}^4+\|\mathcal{E}^\epsilon(t)\|_{s}^{2(s+1)})
\end{align}
for any $\eta_{10}>0$, where $\alpha_1=(1,0,0)$ or $(0,1,0)$ or $(0,0,1)$.  Similarly, we have
\begin{align}\label{ht11}
  \big|\mathcal{S}^{(7)}  \big| +  \big|\mathcal{S}^{(8)}  \big|
 \leq &\eta_{11} \|\nabla \pa \Theta^\epsilon\|^2+C_{\eta_{11}}(\|\mathcal{E}^\epsilon(t)\|_{s}^2+\|\mathcal{E}^\epsilon(t)\|_{s}^{2(s+1)})
\end{align}
for any $\eta_{11}>0$.

For the ninth term $\mathcal{S}^{(9)}$, similar to $\mathcal{R}^{(4)}$, we have
\begin{align}\label{ht12}
  \big|\mathcal{S}^{(9)}  \big|& \leq  \sigma_0\| \partial^\alpha_x\F \| ^2+C(\| \mathcal{E}^\epsilon(t)\| _{s}^s
+1)(\| \partial^\alpha_x
N^\epsilon\| ^2+\| \partial^\alpha_x\Theta^\epsilon\| ^2).
\end{align}

for any $\sigma_0>0$. For the tenth term $\mathcal{S}^{(10)}$, we have
\begin{align}\nonumber
   \mathcal{S}^{(10)}& =  \left\langle\frac{1}{N^\epsilon+\rho^0}\partial^\alpha_x (\Theta^\epsilon)^4,
   \partial^\alpha_x\Theta^\epsilon \right\rangle
+  \big\langle \mathcal{H}^{(10)},
\partial^\alpha_x\Theta^\epsilon
\big\rangle,
\end{align}
where
\begin{align*}
 \mathcal{H}^{(10)} =\partial^\alpha_x\left\{
\frac{(\Theta^\epsilon)^4 }{N^\epsilon+\rho^0}\right\}-\partial^\alpha_x (\Theta^\epsilon)^4
\frac{1}{N^\epsilon+\rho^0}.
\end{align*}
By the positivity of $ \Theta^\epsilon+\rho^0$,
 and the properties of $\Theta^\epsilon$ obtained in Proposition \ref{Pb}, we have
 \begin{align}\nonumber
   \left | \left\langle\frac{1}{N^\epsilon+\rho^0}\partial^\alpha_x (\Theta^\epsilon)^4,
   \partial^\alpha_x\Theta^\epsilon \right\rangle\right|
   \leq C(\| \mathcal{E}^\epsilon(t)\| _{s}^{8}
+\| \partial^\alpha_x\Theta^\epsilon\| ^2).
\end{align}
By the Cauchy-Schwarz, Moser-type inequalities and Sobolev embedding theorem, we infer that
 \begin{align}\nonumber
 &  \ \ \    \big|\big\langle\mathcal{H}^{(10)},
  \partial^\alpha_x\Theta^\epsilon\big\rangle\big| \leq  \big\| \mathcal{H}^{(10)}\big\| \cdot \| \partial^\alpha_x\Theta^\epsilon \|  \nonumber\\
  & \leq C \left[\left\Vert D_x^1\left(\frac{1}{N^\epsilon+\rho^0}\right)\right\Vert_{L^\infty}\| (\Theta^\epsilon)^4 \| _{s-1}
   +\| (\Theta^\epsilon)^4\| _{L^\infty}\left\Vert\frac{1 }{N^\epsilon+\rho^0}\right\Vert_{s}\right]\| \partial^\alpha_x\Theta^\epsilon \| \nonumber\\
 &\leq   C [ (\| \mathcal{E}^\epsilon(t)\| ^{2s}_{s}+\| \mathcal{E}^\epsilon(t)\|^{8}_{s}+1)\| \mathcal{E}^\epsilon(t)\| ^{2}_{s}+ \| \partial^\alpha_x   \Theta^\epsilon\| ^2]. \nonumber
    \end{align}
  Hence, we obtain
   \begin{align}\label{ht13}
  \big| \mathcal{S}^{(10)}\big|
   \leq  C [ (\| \mathcal{E}^\epsilon(t)\| ^{2s}_{s}+\| \mathcal{E}^\epsilon(t)\|^{8}_{s}+1)\| \mathcal{E}^\epsilon(t)\| ^{2}_{s}+
   \|\mathcal{E}^\epsilon(t)\| _{s}^{8}+\| \partial^\alpha_x   \Theta^\epsilon\| ^2].
\end{align}

  The terms $\mathcal{S}^{(11)}$,  $\mathcal{S}^{(12)}$ and  $\mathcal{S}^{(13)}$ can be bounded in a way similar to
  that for $\mathcal{S}^{(10)}$, and hence we get
     \begin{align}\label{ht14}
  \big| \mathcal{S}^{(11)}\big|+\big| \mathcal{S}^{(12)}\big|+\big| \mathcal{S}^{(13)}\big|
   \leq C[(\| \mathcal{E}^\epsilon(t)\| _{s}^{8}+1)\| \partial^\alpha_x\mathcal{E}^\epsilon(t)\| ^2+\|\mathcal{E}^\epsilon(t)\| _{s}^{8}].
\end{align}

Substituting  \eqref{ht1}--\eqref{ht14} into \eqref{HT1}, we conclude that
 \begin{align}\label{HT202}
& \frac12\frac{\dif}{\dif t} \langle \partial^\alpha_x\Theta^\epsilon,
\partial^\alpha_x\Theta^\epsilon \rangle +
\kappa \left\langle  \frac{1}{N^\epsilon+\rho^0}\nabla\partial^\alpha_x \Theta^\epsilon , \nabla\partial^\alpha_x\Theta^\epsilon\right\rangle \nonumber\\
& -(\eta_8+  \eta_{9}+\eta_{10}+\eta_{11})\| \nabla
\Theta^\epsilon \|_s^2  \nonumber\\
\leq   &  {C}_{{\eta}}\big[(\| \mathcal{E}^\epsilon(t)\| ^{2(s+1)}_{s}
    +\| \mathcal{E}^\epsilon(t)\| ^8_{s}+1) \| \mathcal{E}^\epsilon(t)\| ^2_{s} \big]\nonumber\\
  & +\sigma_0\| \F\| ^2_{s}+\eta_7 \|\nabla \partial^\alpha_x \U^\epsilon\|^2
\end{align}
for some constant $ {C}_{{\eta}}>0$ depending on $\eta_i$ ($i=8,9,10,11$).

Applying the operator $\partial^\alpha_x$ to  \eqref{raderror3} and \eqref{raderror4}, multiplying  the resulting equations
    by ${ }\partial^\alpha_x\F $ and   $\partial^\alpha_x\G $  respectively, and integrating then over
    ${\mathbb T}^n$, one obtains that
\begin{align}
  &  \frac12\frac{\dif}{\dif t}(\epsilon \|(\partial^\alpha_x\F , \partial^\alpha_x\G )\| ^2) +  \| \partial^\alpha_x\F \| ^2
   +\| \partial^\alpha_x\G \| ^2\nonumber\\
 =  &   \left\langle\partial^\alpha_x\big[(\Theta^\epsilon)^4\big], \partial^\alpha_x\F \right\rangle
  +4\left\langle\partial^\alpha_x\big[(\Theta^\epsilon)^3\theta^0\big],  \partial^\alpha_x\F \right\rangle\nonumber\\
    &   +6\left\langle\partial^\alpha_x\big[(\Theta^\epsilon)^2(\theta^0)^2\big],  \partial^\alpha_x\F \right\rangle
    +4\left\langle\partial^\alpha_x\big[\Theta^\epsilon(\theta^0)^3\big],  \partial^\alpha_x\F \right\rangle\nonumber\\
 & -  \left\langle ({\epsilon} \partial^\alpha_x\partial_t   \mathbf{q}^0+\epsilon\partial^\alpha_x\partial_t(-\Delta)^{-1}\dv \mathbf{q}^0),
 \partial^\alpha_x\F \right\rangle\nonumber\\
: =  & \sum^{5}_{i=1}\mathcal{T}^{(i)} . \label{L2MM}
\end{align}

For the term $\mathcal{T}^{(1)}$, by Cauchy-Schwarz's inequality, \eqref{mo}, and Sobolev's embedding theorem, we see that
\begin{align}\label{jja}
\big| \mathcal{T}^{(1)}\big|\leq \sigma_1 \|\partial^\alpha_x \F\|^2
  +C_{\sigma_1}\| \mathcal{E}^\epsilon(t)\| _{s}^{8}.
\end{align}
 for any $\sigma_1>0$. Similarly, we have
 \begin{align}\label{jjb}
\big| \mathcal{T}^{(2)}\big|+\big| \mathcal{T}^{(3)}\big|+\big| \mathcal{T}^{(4)}\big|\leq \sigma_2 \|\partial^\alpha_x \F\|^2
  +C_{\sigma_2}(\| \mathcal{E}^\epsilon(t)\| _{s}^{6}+\| \mathcal{E}^\epsilon(t)\| _{s}^{2}).
\end{align}
 for any $\sigma_2>0$.

For the term $\mathcal{T}^{(5)}$, one gets from the regularity of $\dv \mathbf{q}^0$ and Cauchy-Schwarz's inequality that
\begin{align}\label{jjc}
\big| \mathcal{T}^{(5)}\big|\leq C\epsilon^2 +C \|\partial^\alpha_x \F\|^2.
\end{align}

Pulling \eqref{jja}--\eqref{jjc} into \eqref{L2MM}, we find that
\begin{align}
   &  \frac12\frac{\dif}{\dif t}(\epsilon \|(\partial^\alpha_x\F , \partial^\alpha_x\G )\| ^2) + (1-\sigma_1-\sigma_2) \| \partial^\alpha_x\F \| ^2
   +\| \partial^\alpha_x\G \| ^2\nonumber\\
     &\qquad  \leq C(\| \mathcal{E}^\epsilon(t)\| _{s}^{8}+\| \mathcal{E}^\epsilon(t)\| _{s}^{2})+C\epsilon^2. \label{L2M2b}
\end{align}

Combining \eqref{HN1},  \eqref{HU202}, and  \eqref{HT202} with \eqref{L2M2b}, summing up $\alpha$ with
$0\leq|\alpha|\leq s$, using the fact that $N^\epsilon +\rho^0\geq \hat N+\hat \rho>0$,
 choosing $\eta_i$ ($i=1,\dots,11$), and then $\sigma_0, \sigma_1,\sigma_2$ sufficiently small,
and noticing that $s>n/2+2$ is an integer, we obtain \eqref{H2}. This completes the proof of Lemma 3.2.
\end{proof}

With the estimate  \eqref{H2} in hand, we can now  prove Proposition \ref{P31}.

\begin{proof}[Proof of Proposition \ref{P31}]
As in \cite{JL,JL2}, we introduce an $\epsilon$-weighted energy functional
$$\Gamma^\epsilon(t)  =  \v \mathcal{E}^\epsilon(t)\v ^2_{s}.$$
 Then, it follows from \eqref{H2} that there exists a constant
 $\epsilon_1> 0$ depending only on $T$, such that
for any $\epsilon\in (0,\epsilon_1]$ and any  $t\in (0,T]$,
\begin{align}\label{gma}
\Gamma^\epsilon(t)\leq C\Gamma^\epsilon(t = 0)+C\int^t_0\Big\{
\big((\Gamma^\epsilon)^{2(s+1)} +1\big)\Gamma^\epsilon\Big\}(\tau)\dif \tau
+ C\epsilon^2.
\end{align}
 Thus, applying the Gronwall lemma to \eqref{gma}, and keeping in mind that $\Gamma^\epsilon \,(t=0)\leq C\epsilon^2$
 and Proposition \ref{P31}, we find that there exist a $0<T_1<1$ and an $\epsilon>0$, such that $T^\epsilon\geq T_1$
for all $\epsilon\in (0,\epsilon_1]$ and $\Gamma^\epsilon(t)\leq C\epsilon^2$ for all $ t \in (0,T_1]$.
Therefore, the desired a priori estimate \eqref{www} holds. Moreover,
by the standard continuation induction argument, we can extend $T^\epsilon\geq T_0$  to any $T_0<T_*$.
\end{proof}


\medskip \noindent
{\bf Acknowledgements:}
We would like to express our thanks to the anonymous referees for their valuable comments which lead to substantial improvements of the original
manuscript. The first author is supported by the National Basic Research Program under the Grant 2011CB309705, NSFC (Grant Nos. 11229101, 11371065),
and Beijing Center for Mathematics and Information Interdisciplinary Sciences. The second author is supported in part by NSFC (Grant No. 11271184) and PAPD. The third author is supported by NSFC (Grant No.11171213), Shanghai Rising Star Program No.12QA1401600 and Shanghai Committee of Science and Technology (Grant No. 15XD1502300)
.

\bibliographystyle{plain}

\end{document}